\documentclass[12pt]{article}
\RequirePackage[OT1]{fontenc}
\RequirePackage{amsthm,amsmath}
\RequirePackage[numbers]{natbib}
\RequirePackage[colorlinks,citecolor=blue,urlcolor=blue]{hyperref}
\usepackage{amsfonts}
\usepackage{color}
\usepackage{pifont}
\usepackage{amssymb}
\usepackage[english]{babel}
\usepackage[T1]{fontenc}
\usepackage{graphicx}
\usepackage{subfigure}
\usepackage{latexsym}
\usepackage{amstext}
\usepackage{mathrsfs}
\usepackage{hyperref}
\usepackage{dsfont}
\usepackage{graphics}
\usepackage{enumerate}
\usepackage{paralist}
\usepackage{color}
\usepackage{fullpage}
\usepackage{wrapfig}

\newtheorem{thm}{Theorem}
\newtheorem{cor}{Corollary}
\newtheorem{prop}{Proposition}
\newtheorem{rem}{Remark}
\newtheorem{lem}{Lemma}

\newcommand{\dd}{{\rm d}}
\newcommand{\W}{\boldsymbol{W}}
\newcommand{\M}{\boldsymbol{M}}
\newcommand{\D}{\boldsymbol{D}}

\newcommand{\I}{\boldsymbol{I}}
\newcommand{\Q}{\boldsymbol{Q}}

\newcommand{\Z}{\boldsymbol{Z}}
\newcommand{\bplus}{\boldsymbol{\beta}^+}
\newcommand{\bPsi}{\boldsymbol{\Psi}}
\newcommand{\bDelta}{\boldsymbol{\Delta}}
\newcommand{\bLambda}{\boldsymbol{\Lambda}}
\newcommand{\HH}{\boldsymbol{H}}
\begin{document}

\title{An optimal stopping problem for spectrally negative Markov additive processes}

\author{  M. \c{C}a\u{g}lar\footnote{ College of Sciences,  Department of Mathematics,
                  Ko\c{c} University,
                  Rumeli Feneri Yolu,
                  34450 Sariyer, Istanbul
                   Turkey,
                  \texttt{mcaglar@ku.edu.tr }} \and A. Kyprianou\footnote{Department of Mathematical Sciences, University of Bath, Claverton Down, BA2 7AY, UK,\texttt{a.kyprianou@bath.ac.uk}}
                 \and
                 C. Vardar-Acar\footnote{Middle East Technical University, Department of Statistics,  \"{U}niversiteler Mah. Dumlupınar Blv. No:1, 06800 \c{C}ankaya Ankara, Turkey, \texttt{cvardar@metu.edu.tr}}}

\maketitle

\begin{center}
{\it In  memory of Larry Shepp}
\end{center}

\begin{abstract}
Previous authors have considered optimal stopping problems driven by the running maximum of a spectrally negative L\'evy process as well as  of a one-dimensional diffusion; see e.g. \cite{KO, bottleneck, O, LP, guo_shepp, pedersen, gapeev}. Many of the aforementioned results are either implicitly or explicitly dependent on Peskir's maximality principle, cf. \cite{maximality_principle}. In this article, we are interested in understanding how some of the main ideas from these previous works can be brought into the setting of problems driven by the maximum of a class of Markov additive processes (more precisely Markov modulated L\'evy processes).
Similarly to \cite{O, KO, bottleneck}, the optimal stopping boundary is characterised by a system of ordinary first-order differential equations, one for each state of the modulating component of the Markov additive process. Moreover, whereas scale functions played an important role in the previously mentioned work,  we work instead with scale matrices for Markov additive processes here; as introduced by \cite{kypal, ivan}.
We exemplify our calculations in the setting of the Shepp-Shiryaev optimal stopping problem \cite{shep1, shep2}, as well as a family of capped maximum optimal stopping problems.

\smallskip

\noindent MSC Classification: 60G40

\smallskip

\noindent Key words: Optimal stopping, Scale matrices, excursion theory, Markov additive processes.
\end{abstract}

\section{Introduction}
We are interested in a family of optimal stopping problems driven by a class of Markov additive processes (MAPs) that have frequently appeared  in the applied probability literature. The aforesaid optimal stopping problems are based around  maximum functionals of such MAPs. In order to make precise the setting in which we want to work, we will start by defining the class of Markov additive processes that we are interested in.
\bigskip

Let $E$ be a finite state space and, hence, without loss of generality, we can write it in the form $\{1,\cdots, N\}$.
A c\`adl\`ag process $(X,J)$ in $\mathbb{R} \times E$
with probabilities  $\mathbb{P}_{(x,i)} = \mathbb{P}( \cdot \,\vert\, X_0 = x, J_0 = i)$, $x\in \mathbb{R}$, $i\in E$, is called a
{\it Markov additive process (MAP)}
if 
for any $i \in E$, $s,t \ge 0$ and bounded and measurable $f$, on the event $\{J_t = j\}$
\begin{eqnarray}
\label{e:MAP}
\mathbf{E}_{x,i}[f(X_{t+s}-X_{t}, J_{t+s})| \mathcal{G}_t]
& = \mathbf{E}_{0,j}[f(X_s, J_s)]
\end{eqnarray}
where  $(\mathcal{G}_{t},t\geq 0)$
 is the filtration generated by $(X,J)$. 
The process $J$ is thus a Markov chain on $E$ and is called the {\it modulator} of $X$, whereas   the latter is called the {\it ordinator}.


\bigskip

Below, we will briefly discuss some relevant  aspects of the theory of MAPs for our purposes. The reader is principally referred to
\cite{Asm-rp1} and  \cite[\S XI.2a]{Asmussen} for further details. Older literature includes \cite{Cinlar1, Cinlar2, AS, AK}.
We will mainly appeal to the setup and  notation of
\cite{Iva-thesis},
where it was principally assumed that
$X$ is spectrally negative (only negative jumps).

%
%

\bigskip

It turns out that
  the pair $(X,J)$ is a Markov additive process if and only if, for each $i,j\in E$,
  there exist a sequence of iid L\'evy processes
  $(X^{i, (n)}, n\ge 0)$ and  a sequence of iid random variables
  $(U_{i,j}^n, n\ge 0)$, independent
  of the chain $J$, such that if $\sigma_0 = 0$
  and $(\sigma_n,n \ge 1)$ are the
  jump times of $J$, the process $X$ has the representation
  \[ X_t = \mathbf{1}_{(n > 0)}( X_{\sigma_n -} + U_{J_{\sigma_n-}, J_{\sigma_n}}^n) + {X}^{J_{\sigma_n\!\!} \, ,(n)}_{t-\sigma_n},
\text{ for }t \in [\sigma_n, \sigma_{n+1}),\, n \geq 0. \]

For each $i \in E$, it will be convenient to define
 $X^i$ as a L\'evy process whose law is that of the $X^{i,(n)}$ processes in the above representation; and similarly, for each $i,j \in E$, define $U_{i,j}$ to
  be a random variable having the common law of the $U_{i,j}^n$ variables. From the above representation one may deduce that  MAPs are strong Markov processes. Indeed, they satisfy the slightly stronger property that
\eqref{e:MAP} holds with $t$ replaced by a stopping time, albeit on the event that the stopping time is finite.

\bigskip

Henceforth, we confine ourselves
to irreducible (and hence ergodic) Markov chains $J$.
Let the state space $E$ be the finite set $\{1, \ldots, N\}$, for some $N \in \mathbb{N}$. {\color{black}Moreover, we will additionally assume that each of the processes $X^i$ are spectrally negative L\'evy processes, allowing for the possibility that some of them have monotone increasing paths (but disallowing the possibility that any of them have purely non-increasing paths\footnote{\color{black}As noted in e.g. \cite{ivan}, we can  incorporate the case that some, but not all, of the $X^i$ have non-increasing paths by performing a time change and compressing each section of path that is non-increasing into a single negative jump.}), so long as the MAP may experience both upwards and downwards movements.}  Similarly, we will also assume that $U_{i,j}\leq 0$  for each $i,j\in E$. As such, the process $(X,J)$ is said to be  a spectrally negative MAP.
Note that this assumption ensures that $(X,J)$, when issued from $X_0 = 0$ and $J_0 = i\in E$, must have the property that 0 is regular for $(0,\infty)$. What is less clear is whether 0 is regular for $(-\infty,0)$ as well. This depends on the state $i\in E$. If  $X^{i}$ is of unbounded variation, then 0 is regular for $(-\infty,0)$  for $X$ when $X_0 = 0$, $J_0 = i$, otherwise, when there is bounded variation, there is an almost surely strictly positive time before $X$ enters $(-\infty,0)$, that is,  0 is irregular for $(-\infty,0)$ for $X$ \cite[pg.232]{kbook}.

\bigskip

Denote the transition rate matrix of the chain $J$ by
${\boldsymbol{Q}} = (q_{i,j})_{i,j \in E}$.
For each $i \in E$, the Laplace exponent of the L\'evy process $X^i$
will be written $\psi_i$. 
For each pair of $i,j \in E$,
define
the Laplace transform $G_{i,j}(z) = \mathbb{E}({\rm e}^{z U_{i,j}})$
 of the
distribution of the jump $U_{i,j}$,
where this exists. Write ${\boldsymbol{G}}(z)$ for the $N \times N$ matrix
whose $(i,j)$th element is $G_{i,j}(z)$. 
We will adopt the convention that $U_{i,j} = 0$ if
$q_{i,j} = 0$, $i \ne j$, and also set $U_{i,i} = 0$ for each $i \in E$.
\bigskip

The multidimensional analogue of the Laplace exponent of a L\'evy process is
provided by the matrix-valued function
\begin{equation}\label{e:MAP F}
 {\bPsi}(z) = {\rm diag}( \psi_1(z), \ldots ,\psi_N(z))
  + {\boldsymbol{Q}} \circ {\boldsymbol{G}}(z),
\end{equation}
for all $z \in \mathbb{C}$ where the elements on the right are defined,
where $\circ$ indicates elementwise multiplication.
It is  known
that
\begin{equation} \mathbb{E}_{(0,i)}[ {\rm e}^{z X_t} ; J_{t}=j] = \bigl({\rm e}^{\bPsi(z) t}\bigr)_{i,j} , \text{ for } i,\,j \in E,  t\geq 0,
\label{psi2}
\end{equation}
for all $z \in \mathbb{C}$ where one side of the equality is defined \cite[Prop.XI.2.2]{Asmussen}.
For this reason, $\bPsi$ is called the {\it matrix exponent} of
the MAP $(X, J)$.
Note  that $\bPsi(z)$ is well defined and finite at least for
${\rm Re}(z)\geq 0$.

\bigskip

For $z$ such that $\bPsi(z)$ is well defined,  there exists a  leading real-valued eigenvalue of the matrix $\bPsi(z)$, also called the
{\it Perron--Frobenius eigenvalue};
see \cite[\S XI.2c]{Asmussen} and \cite[Proposition 2.12]{Iva-thesis}.
If we denote this eigenvalue by $\kappa(z)$, then it turns out that it
 is larger than the real part of all its other eigenvalues.
Furthermore, the
corresponding
right-eigenvector $\boldsymbol{v}(z)$ has strictly positive entries,
and can be normalised such that
$  \boldsymbol\pi \cdot \boldsymbol{v}(z) = 1$,
where we recall that $\boldsymbol\pi$ is the stationary distribution of
the underlying chain $J$.

\bigskip
The
eigenvalue $\kappa \left( \theta \right) $ is a convex function on $(0,\infty)$  such that $\kappa \left( 0\right) =0$ and the derivative $\kappa'(0+)$  exists in $[-\infty,\infty)$.  A trichotomy similar in spirit to the one that describes the long term behaviour of L\'evy processes exists, which states that either $\lim_{t\to\infty}X_t = \infty$, $\lim_{t\to\infty}X_t = -\infty$
or $\limsup_{t\to\infty}X_t = -\liminf_{t\to\infty}X_t = \infty$ accordingly as $\kappa'(0+)>0$,   $\kappa'(0+)<0$ or $\kappa'(0+)=0$, respectively.
For
the right inverse of $\kappa$ we shall write $\Phi$. That is, for all $q\geq 0$,
\begin{equation}
\Phi(q) = \sup\{\theta \geq 0: \kappa(\theta)=q\}.
\label{kinverse}
\end{equation}
The properties of $\kappa$ imply that $\Phi(q)>0$ for $q>0$ and $\Phi(0) = 0$ if and only if $\kappa'(0+)\geq 0$, otherwise $\Phi(0)>0$.
\bigskip

The eigenvalue $\kappa$ also affords us the opportunity to  introduce the natural analogue of the Esscher transform for MAPs.  Specifically, for $t\geq 0$,
\begin{equation}
\left. \frac{\dd\mathbb{P}_{(x,i)}^{\gamma }}{\dd\mathbb{P}_{(x,i)}}\right| _{%
\mathcal{G}_{t}}:={\rm e}^{\gamma (X\left( t\right) -x)-\kappa \left( \gamma
\right) t}\frac{v_{J\left( t\right) }\left( \gamma \right) }{v_{i}\left(
\gamma \right) },
 \label{alaGirsanov}
\end{equation}
for $\gamma$ such that $\kappa(\gamma)<\infty$. (The most common use of \eqref{alaGirsanov} in this article will be  when we take the value $\gamma = \Phi(q)$, for $q\geq 0$.) The process $(X,\mathbb{P})$ is again a spectrally
negative MAP whose intensity matrix $\bPsi_\gamma\left(
\theta \right) $ is well defined and finite for $\theta \geq
-\gamma $. If $\bPsi_\gamma\left( \theta \right) $ has
largest eigenvalue $\kappa _{\gamma }\left( \theta \right) $ and
associated right eigenvector $\boldsymbol{v}_{\gamma }\left( \theta
\right) $, the triple $\left( \bPsi_\gamma\left( \theta
\right) ,\kappa _{\gamma }\left( \theta \right)
,\boldsymbol{v}_{\gamma }\left( \theta \right) \right) $ is related to
the original triple $\left(
\bPsi\left( \theta \right) ,\kappa \left( \theta \right) ,\boldsymbol{v}%
\left( \theta \right) \right) $ via
\begin{equation}
\bPsi_\gamma\left( \theta \right) =\bDelta _{\boldsymbol{v}}\left( \gamma
\right) ^{-1}\bPsi\left( \theta +\gamma \right) \bDelta _{\boldsymbol{v}%
}\left( \gamma \right) -\kappa \left( \gamma \right) \I\quad\text{ and }\quad%
\kappa _{\gamma }\left( \theta \right) =\kappa \left( \theta
+\gamma \right) -\kappa \left( \gamma \right)\;,
\label{meas-change-effect}
\end{equation}
where $\I$ is the $N\times N$ identity matrix and
\begin{equation*}
\bDelta _{\boldsymbol{v}}\left( \gamma \right) :=\text{diag}\left( v_{1}\left(
\gamma \right) ,\ldots ,v_{N}\left( \gamma \right) \right) .
\end{equation*}
Note in particular, for the choice $\gamma = \Phi(q)$, it is easy to verify that $\kappa_{\Phi(q)}'(0+)>0$ and hence $(X,\mathbb{P}^{\Phi(q)})$ drifts to $+\infty$.

\bigskip
These details and more concerning the basic characterization of
MAPs can be found in \cite[Chp.XI]{Asmussen}. See also \cite{ivan, Iva-thesis}.
\bigskip

Define $\overline{X}_t= \sup_{s\leq t}X_s$, $t\geq 0$ and let $G_t = \sup\{s\leq t:  X_s = \overline{X}_t\}$, from which we can define $\bar{J}_t = J_{G_t}$, $t\geq 0$. The quadruple $( X,\overline{X}, J, \bar{J} )$ is also a strong Markov process with respect to $(\mathcal{G}_t, t\geq 0)$, which is an important fact that will drive the analysis in this paper. To see this, we can write the process $( X,\overline{X}, J, \bar{J} )$ with additional indices, as $( X^{(x)},\overline{X}^{(s)}, J^{(i)}, \bar{J}^{(j)} )$ indicating its point of issue as the state $(x, s, i, j)$ where $x\leq s$ and $i,j\in E$.  For any  $\mathcal{G}_t$-stopping time  $\tau$, on $\{\tau<\infty\}$, we have for bounded, measurable $F: \mathbb{R}^2\times E^2\to[0,\infty)$, $\mathbb{E}\left[\left.
F(X^{(x)}_{\tau+t},
\overline{X}^{(s)}_{\tau+t},
J^{(i)}_{\tau+t},
\bar{J}^{(j)}_{\tau+t})\right|\mathcal{G}_\tau\right]\notag $ is equal to
\begin{align*}
\mathbb{E}_{(0,i')}\left[
F\left(
x' +  X_t,\,
  s'\vee(x'+ {\bar{X}} _t),\,
{J}_t,\,
  j' \mathbf{1}_{( s'\geq (x'+ \bar{X}_t))} + {J}_{G_t} \mathbf{1}_{( s'<(x'+ \bar{X}_t))}
\right)
\right]\, ,
\end{align*}
where $(x',s', i', j') = (X_\tau^{(x)}, \overline{X}_\tau^{(s)}, J_\tau^{(i)}, \bar{J}_\tau^{(j)})$. 
\bigskip

Henceforth, we refer to the quadruple $( X,\overline{X}, J, \bar{J} )$ as the \emph{Markov additive maximality  process}   (MAMP).  As such, we abuse our earlier notation and write $\mathbb{P}_{(x,s, i, j)}$ to denote the law conditional on $(X_0,\overline{X}_0, {J}_0, \bar{J}_0) = (x,s, i, j)$, for $x\in\mathbb{R}$, $s\geq x$, $i,j\in E$. The reader will note the deliberate abuse of notation, albeit being consistent with $\mathbb{P}_{(x,i)}$, for $x\in \mathbb{R}$ and $i\in E$. Writing $\mathbb{P} = (\mathbb{P}_{(x,s, i, j)}, x\leq s, i,j\in E)$, we will refer to the underlying MAP as $((X,J),\mathbb{P})$.
\bigskip

Let us return to the family of optimal stopping problems that we are interested in. Fundamentally, we want to consider  problems driven by the MAMP  that take the form
\begin{equation}
V(x,s,i,j)=\sup_\tau \mathbb{E}_{(x,s,i,j)}[{\rm e}^{-q\tau}f(\overline{X}_\tau,\bar{J}_\tau)], \qquad s\geq x, i,j\in E,
\label{OST}
\end{equation}
where the supremum is taken over the class of almost surely finite $\mathcal{G}_t$-stopping times, $q\geq 0$ and  $f: \mathbb{R}\times E \to (0,\infty)$ is a measurable function.

\bigskip

Previous authors have considered optimal stopping problems driven by the running maximum of a spectrally negative L\'evy process $X$, as well as in setting of a general diffusion; see e.g. \cite{KO, bottleneck, O, LP, guo_shepp, pedersen, gapeev}. Many of the aforementioned results are either implicitly or explicitly dependent on Peskir's maximality principle, cf. \cite{maximality_principle}. In this article, we are interested in understanding how some of the main ideas from these previous works can be brought into the setting described above, albeit using a heuristic developed in Peskir \cite{maximality_principle} and in the PhD thesis of C. Ott \cite{curdin}.
\bigskip

Whereas there are several works concerning optimal stopping problems with regime switching, see e.g.  \cite{guo, DH}, we believe this is the first such work which considers the current setting of MAPs. What is also new in the current setting is that we make use in our analysis of the so-called {\it scale matrices}, introduced in \cite{kypal, ivan}. Furthermore, we introduce an alternative second scale matrix in this paper. These are matrix-valued functions which play a similar role of scale functions in the theory of spectrally negative L\'evy processes (a well known tool in the setting of a number of classical applied probability models, cf. \cite{KKR, kbook}), albeit in the setting of the family of MAPs described above.

\bigskip

The rest of this paper is organised as follows. In the next section we consider the formal definition of a scale matrix and look at some of its analytical properties. In Section \ref{SectOST}, we describe the solution to \eqref{OST} and highlight the strategy that we will take, which is based on what we call {\it Peskir--Ott heuristic}. Moreover, we pass through a number of technical results which allows us to prove our main result based on a standard verification technique.  In Section 4, we exemplify our calculations using the Shepp-Shiryaev optimal stopping problem. Finally, Section 5 concludes the paper and sets an outlook on a family of capped maximum optimal stopping problems.

\section{Scale matrices}
Theorem 3 of \cite{kypal} (see also Theorem 1 of \cite{ivan}) introduces a family of  $N\times N$ matrix function $\W^{(q)}: \mathbb{R}\to\mathbb{M}(\mathbb{R})$, the space of $N\times N$ matrices with real-valued entries, for $q\geq0$, such that $\W^{(q)}(x) ={\boldsymbol 0}$ (the zero matrix) for $x<0$ and otherwise is an almost-everywhere differentiable, non-decreasing function which is defined via the Laplace transform
$$\int^{\infty}_{0}{\rm e}^{-\beta x}\W^{(q)}( x)\dd x=(\bPsi(\beta)-q\I)^{-1}$$
for $\beta >\max\{{\rm Re}(z): z\in \mathbb{C}, {\rm det}(\bPsi(z)) =q \}$.
An important feature of the scale matrix is that it plays a fundamental role in several key fluctuation identities, very much in the spirit of scale functions in the setting of spectrally negative L\'evy processes (cf. Chapter 8 of \cite{kbook} and \cite{KKR}). This was first discussed in detail in \cite{kypal}.

\bigskip

A second scale function was introduced in \cite{kypal}, $\M^{(q)}$, and  this function was presented as the natural analogue of the second scale function that appears also in the theory of spectrally negative L\'evy processes.
In this section, we introduce an alternative second scale matrix (of course closely related to $\M^{(q)}$), which mirrors more clearly the situation for the second scale function in the theory of spectrally negative L\'evy processes; see Chapter 8 of \cite{kbook} or  \cite{KKR} for comparison. We define  the $N\times N$ matrix functions $\Z^{(q)} $ on $ \mathbb{R}$, for $q\geq0$,  by
\begin{equation}
\Z^{(q)}(x)=\I+\int^{x}_{0}\W^{(q)}(y){\dd y}~(q\I-\Q), \qquad x\in\mathbb{R}.
\label{Z}
\end{equation}
Let $1\!\!1$ denote the vector $(1,\cdots, 1)^T\in \mathbb{R}^N$. Note that
\begin{equation}
[\Z^{(q)}(x)1\!\!1]_i=1+q\int^{x}_{0}[\W^{(q)}(y)1\!\!1]_i{\dd y}, \qquad x\in\mathbb{R}, i\in E.
\label{Zq1}
\end{equation}
This definition is consistent with the usage of the notation $Z^q$ in \cite[pg.1165]{JI} and coincides with (3) in \cite{ivan} for $\alpha=0$ and $q=0$.
\bigskip

Let us define $\tau^+_a= \inf\{t>0 : X_t>a\}$ and $\tau^-_0 = \inf\{t>0: X_t<0\}$.
 Considering \eqref{alaGirsanov} with $\gamma= \Phi(q)$, for $q\ge 0$, we denote by $\bLambda(q)$ the  intensity matrix of the modulator of the ascending ladder MAP under $\mathbb{P}^{\Phi(q)}$  (cf. Section 3 of \cite{kypal}). That is, $\mathbb{P}^{\Phi(q)}_{0,i}(J_{\tau^+_a}=j)= [{\rm e}^{a\bLambda(q)}]_{i,j} $. Note that $((X,J),\mathbb{P}^{\Phi(q)})$ is a process for which the ordinate drifts to infinity and hence $\bLambda(q)$ is the intensity matrix of a conservative Markov chain {\color{black} (the alternative being that the chain can be killed and sent to a cemetery state at a rate that depends on the current state)}. The following result gives existing identities that can be found in the literature, from \cite{JI, ivan, kypal} (precise references in its proof). In particular, the approach in \cite{JI, ivan} allows for the slightly more general setting that discounting rate $q$ is replaced by one that  is state dependent.

\begin{prop}\label{prop1}  We have for $x\leq a$, $i,j\in E$ and $q\geq 0$, the three identities

\begin{equation}
\label{id0}
\mathbb{E}_{(x,i)}[{\rm e}^{-q\tau^{+}_{a}};\tau_{a}^{+}<\infty,J_{\tau_{a}^{+}}=j]=[\bDelta _{\boldsymbol{v}}\left( \Phi(q) \right)
{\rm e}^{(a-x) \bplus(q)
}
\bDelta _{\boldsymbol{v}}\left(  \Phi(q)  \right)^{-1}]_{i,j} ,
\end{equation}

\begin{equation}
\label{id1}
\mathbb{E}_{(x,i)}[{\rm e}^{-q\tau^{+}_{a}};\tau_{a}^{+}<\tau_{0}^{-},J_{\tau_{a}^{+}}=j]=[\W^{(q)}(x)\W^{(q)}(a)^{-1}]_{i,j},
\end{equation}

\begin{equation}
\label{id2}\mathbb{E}_{(x,i)}[{\rm e}^{-q\tau^{-}_{0}};\tau_{0}^{-}<\tau_{a}^{+},J_{\tau_{0}^{-}}=j]=[\Z^{(q)}(x)-\W^{(q)}(x)\W^{(q)}(a)^{-1}\Z^{(q)}(a)]_{i,j},
\end{equation}
where $\bplus(q): = \bLambda(q)-\Phi(q)\I $ is the matrix exponent of the ascending ladder MAP. 	
\end{prop}

\begin{rem}\label{remQ+}
\rm
It is worth noting that as the MAP $(X,J)$ is spectrally negative, the ascending ladder MAP consists of a pure drift of unit speed, which is (possibly) killed at a rate which depends on its modulating chain (corresponding to the ability of the MAP $(X,J)$ to drift to $-\infty$).
\end{rem}

\begin{proof}[Proof (of Proposition \ref{prop1})] The first identity is taken from
of Theorem 1 of \cite{kypal}.
Identity \eqref{id1} is lifted from Theorem 3 (iv) of \cite{kypal} and appears in \cite[pg.1164]{JI}. For the third identity, we also recall from Theorem 3 that there exists matrix function $\M^{(q)}(x)$, defined  by $\I$ for $x\leq0$ and otherwise, for $x>0$, via its Laplace transform, given by\footnote{There is a typo in the statement of Theorem 3 of \cite{kypal}, the Laplace transform of the measure $\M^{(q)}(\dd x)$  should be taken over $(0,\infty)$.}
\begin{eqnarray}
\label{whatisD}
\int_{(0,\infty)}{\rm e}^{-\beta x}\M^{(q)}(\dd x)&=&(\bPsi(\beta)-q\I)^{-1}(I-\beta\hat{\D}(q)^{T})(q\I-\Q)
\end{eqnarray}
such that
\begin{equation}
\label{id3}\mathbb{E}_{(x,i)}[{\rm e}^{-q\tau^{-}_{0}};\tau_{0}^{-}<\tau_{a}^{+},J_{\tau_{0}^{-}}=j]=[\M^{(q)}(x)-\W^{(q)}(x)\W^{(q)}(a)^{-1}\M^{(q)}(a)]_{i,j}.
\end{equation}
holds for $x\in\mathbb{R}$.
The precise meaning of $\hat{\D}(q)$ will turn out to be unimportant for our proof here; however, the reader can refer to Section 4 of \cite{kypal} otherwise.  

\bigskip

Using the Laplace transform of $\W^{(q)}$ given by
\[
\int_{[0,\infty)}{\rm e}^{-\beta x}\W^{(q)}(x)\dd x = (\bPsi(\beta)-q\I)^{-1}\: ,
\]
and integration by parts yields, for $x\geq 0$,
\[
\M^{(q)}(x)-\I =\int^{x}_{0}\W^{(q)}(y){\dd y}~(q\I-\Q)-\W^{(q)}(x)\hat{\D}(q)^{T}(q\I-\Q),
\]
and hence, substituting this expression for $\M^{(q)}$  into \eqref{id3}, we recover
\begin{align*}
\label{id3}
&(\mathbb{E}_{(x,i)}[{\rm e}^{-q\tau^{-}_{0}};\tau_{0}^{-}<\tau_{a}^{+},J_{\tau_{0}^{-}}=j])_{i,j\in E}\\
&=\Z^{(q)}(x)-\W^{(q)}(x)\W^{(q)}(a)^{-1}\Z^{(q)}(a)\\
&\hspace{2cm}-[\W^{(q)}(x)- \W^{(q)}(x)\W^{(q)}(a)^{-1}\W^{(q)}(a)]\hat{\D}(q)^{T}(q\I-\Q)\\
&=\Z^{(q)}(x)-\W^{(q)}(x)\W^{(q)}(a)^{-1}\Z^{(q)}(a),
\end{align*}
as required. The special case of \eqref{id2} for $q=0$ can also be obtained by taking $\alpha=0$ in \cite[Corollary 3]{ivan}.
\end{proof}

{\color{black}As a small remark following the above proof, we also point to the work of Breuer \cite{Br}, who also addresses results containing special case of  Proposition \ref{prop1} by appealing to generator equations.}

\bigskip

The identities in Proposition \ref{prop1} will prove to be useful, but we will also need to know some basic facts about the smoothness properties of the scale matrices $\W^{(q)}$ and $\Z^{(q)}$. Next, Theorem  \ref{ZWlem}
gathers what we will need to know later on in this paper. Its proof is somewhat technical and therefore deferred to the Appendix. The reader may also consult \cite{ivan} for related results.
Let us denote the states of the modulator for which the ordinate moves as a bounded variation L\'evy process  by $E^{\texttt{bv}}$. Then, $E^{\texttt{ubv}}:= E\backslash E^{\texttt{bv}}$ corresponds to the  states for which we have  an unbounded variation L\'evy process.

\begin{thm}\label{ZWlem}
Fix $q\ge 0$.  \begin{itemize}

\item[(i)] For all $a>0$,  $[\W^{(q)}(0+)\W^{(q)}(a)^{-1}]_{i,j} >0$, for all $j\in E$, if and only if $i\in E^{\emph{\texttt{bv}}}$, otherwise, when $i\in E^{\emph{\texttt{ubv}}}$, we have $ [\W^{(q)}(0+)\W^{(q)}(a)^{-1}]_{i,j} =0$ for all $j\in E$.

Moreover, $\W^{(q)}(0+) =\mathrm{diag}(W^{(q)}_1(0+),\ldots,W^{(q)}_N(0+))$, where $W^{(q)}_i $ is the scale function corresponding to L\'evy process $X^i$, $i\in E$.

\item[(ii)]  The scale matrix $\W^{(q)}(x)$
is almost everywhere differentiable. Moreover, for all $0\leq x\leq a$,  the almost everywhere derivative  of $\W^{(q)}(x)\W^{(q)}(a)^{-1}$ has an (elementwise) strictly positive and right-continuous derivative on $\mathbb{R}$;


\item [(iii)]  For $i\in E^{\emph{\texttt{ubv}}}$, $\W^{(q)\prime}_{ii}(0+)=W^{(q+q_i)\prime}_i (0+)$,
where $W^{(q+q_i)}_i $ is the scale function corresponding to L\'evy process $X^i$, $i\in E$, and $q_i=\sum_{j\neq i} q_{i,j}$.

\item[(iv)]  $\Z^{(q)}$ is continuously differentiable, except possibly at $0$ and otherwise is almost everywhere twice differentiable with a right-continuous second derivative.
\end{itemize}
\end{thm}

\begin{cor}\label{rem2}  For each $i\in E$, there exists $a(i)\in (0,\infty]$ such that $[\Z^{(q)}(x)1\!\!1]_i$ is not smaller than unity for  $0\le x\le a(i)$.
\end{cor}

\begin{proof} By Theorem \ref{ZWlem} (i), $[\W^{(q)}(0+)1\!\!1]_i=W^{(q)}_i(0+)$ for  $i\in E$ and, from Theorem \ref{ZWlem} (iii), we have $[\W^{(q)\prime}(0+)1\!\!1]_i=W^{(q+q_i)\prime}_i(0+)$ for $i\in E^{{\texttt{ubv}}}$. Since $W^{(q)}_i(0+)> 0$  for $i\in E^{{\texttt{bv}}}$ and $W^{(q+q_i)\prime}_i(0+)> 0$ for $i\in E^{{\texttt{ubv}}}$ by \cite[Lem.3.1,Lem.3.2]{KKR},  $W^{(q)}_i(x)$ is strictly positive in a neighborhood of 0 in $\mathbb{R}_+$ for all $i\in E$. More precisely, for all $q \geq 0,$ $W_{i}^{(q) \prime}(0+)$ equals $2 / \sigma_i^{2}$  or  $+\infty$ when $\sigma_{i}=0$ by \cite[Lem.3.2]{KKR} when $i\in E^{{\texttt{ubv}}}$. Then, the result follows as $\Z^{(q)}(x)$ is a continuous function of $x\in \mathbb{R}$ and $[\Z^{(q)}(0)1\!\!1]_i=1$.
\end{proof}

\section{Optimal Stopping Problem solution}\label{SectOST}

The following result gives us a relatively complete solution to \eqref{OST}.

\begin{thm}\label{main}
Suppose that $q>0$ and  $f: \mathbb{R}\times E \to (0,\infty)$ is a measurable function. 
For each measurable $g: \mathbb{R}\times E\to [0,\infty)$ define
\begin{equation}
\tau_g=\inf\{u>0:\overline{X}_{u}-X_u>g(\overline{X}_u,\bar{J}_u)\},
\label{taugdef}
\end{equation}
where $g$ is defined as a non-negative solution to  the first order differential equation
\begin{equation}
g^\prime(s,j) = 1-\frac{ f^\prime(s,j)}{f(s,j)} \frac{[\Z^{(q)} (g(s,j)) 1\!\!1]_{j}}{ [ \Z^{(q)\prime} (g(s,j))1\!\!1]_{j}},\qquad j\in E, s\in\mathbb{R}.
\label{gODE}
\end{equation}
  satisfying $g(s,j)\le a(j)$, where $a(j)= \inf\{x> 0:  [\Z^{(q)}(x)1\!\!1]_j \le 1\}$, $j\in E$.
Assume the following;
\begin{itemize}
\item[(i)] {\color{black}at least one solution to \eqref{gODE} exists;} 
\item[(ii)] the stopping time $\tau_g$ is almost surely finite for $((X,J),\mathbb{P})$;
\item[(iii)] the function  $f(s,j)$ is continuously differentiable for each $j\in E$;
\item[(iv)] for all $t\geq 0$ and $i,j\in E$,
\begin{equation}
\mathbb{E}_{(x,i)}\left[\int_0^t f(\overline{X}_u,\bar{J}_u)  {\rm e}^{\Phi(q)X_u}\dd u\right]+\mathbb{E}_{(x,i)}\left[\int_0^t f(\overline{X}_u,\bar{J}_u)^2  {\rm e}^{2\Phi(q)X_u}\dd u\right]<\infty.
\label{conditions}
\end{equation}
\end{itemize}
Then the optimal stopping problem \eqref{OST} is solved with optimal strategy $\tau_g$ and value function
\begin{equation}
V(x,s,i,j)
=
f(s,j) [\Z^{(q)}(x-s+g(s,j)) 1\!\!1]_{i} \qquad  x\leq s, i,j\in E.
\label{valuefn}
\end{equation}
\end{thm}

\begin{rem}\rm {\color{black}Our insistence that $q>0$ is not necessary and, in principle, one may also consider the case $q = 0$. Additionally, one may also consider the case that discounting at rate $q t$ is replaced by discounting at rate  $\int_0^t \eta_{J_s}\dd s$, where $(\eta_i,i\in E)$ are different rates of discounting depending on the modulating state. The latter is also equivalent to setting $q =0$ and including in the definition of $(X, J)$ the possibility that $J$ is non-conservative. It is likely that this slightly more general structure nonetheless preserves the identities in e.g. Proposition \ref{prop1} as well as the subsequent analytical properties, albeit a careful audit of existing literature being needed. In this respect, we claim that our approach will also work suitably well for optimal stopping problems of the form

\[
V(x,s,i,j)=\sup_\tau \left\{\mathbb{E}_{(x,s,i,j)}[ {\rm e}^{-\int_0^\tau \eta_{J_s}\dd s}f(\overline{X}_\tau,\bar{J}_\tau)]
- \mathbb{E}_{(x,s,i,j)}\left[\int_0^\tau {\rm e}^{-\int_0^s \eta_{J_u}\dd u} c(\overline{X}_u,  \bar{J}_u){\rm d}u \right]\right\},
\]
for $s\geq x$, $ i,j\in E$, and  a suitable choice of $c$, e.g. uniformly bounded and continuous.
Moreover, again with appropriate assumptions, we claim that the methods we present will also allow one to handle the functions $f$ and $c$ depending on the full Markov process  $(X, \overline{X}, J, \bar{J})$.
}
\end{rem}



\subsection{Peskir--Ott Heuristic}\label{heuristic}

We will prove Theorem \ref{main} by appealing to a series of lemmas. However, before dealing with those, let us first introduce some basic reasoning which will heuristically explain the core ingredients of the proof  of Theorem \ref{main}.  We refer to this as the  {\it Peskir--Ott heuristic} following the introduction in his PhD thesis \cite{curdin}, in which he outlines similar reasoning for the solution to a large family of optimal stopping problems driven by the maximum of a spectrally negative L\'evy process, which are analogous to those we consider here in the MAP setting. In turn his reasoning was stimulated by the arguments in Peskir \cite{maximality_principle}.

\bigskip

The point of interest in the current context is that reasoning of \cite{curdin} is equally applicable, on account of the fact that we have identified $(X, \overline{X}, J, \bar{J})$ as the natural driving Markov process to \eqref{OST}.
To this end, let us
guess a solution of the form
\begin{equation}
\tau_g=\inf\{u>0:\overline{X}_u-X_u>g(\overline{X}_u,\bar{J}_u)\},
\label{taug}
\end{equation}
where measurable $g: \mathbb{R}\times E\to [0,\infty)$ is to be determined.
On account of the fact that $f(s,j)=0$ for $s<s^*$, it is natural to set $g(s,j) = +\infty$, for $s< s^*$ and $j\in E$.

The value function associated with strategy $\tau_g$, given by
\begin{equation}
V_g(x,s,i,j)= \mathbb{E}_{(x,s,i,j)}
[{\rm e}^{-q\tau_g}f(\overline{X}_{\tau_g},\bar{J}_{\tau_g})].
\label{Vg}
\end{equation}
Note that  $\tau_g< \tau_s^+$ if and only if $\overline{X}_u<s$ for all $u\le \tau_g$. Moreover, we have
\[
\tau_{g}=\inf\{u>0:s-X_u>g(s,j)\}=\inf\{u>0:X_u<s-g(s,j)\} = \tau^-_{s-g(s,j)},
\]
and hence
$\tau_{g}> \tau_s^+$ if and only if $ \tau^-_{s-g(s,j)}> \tau_s^+$.
\bigskip

Conditioning on the minimum of  $\tau_s^+$ and $\tau_g$, and applying Markov property, we find
\begin{eqnarray}
\lefteqn{V_g(x,s,i,j)=\mathbb{E}_{(x,s,i,j)}\left[ {\rm e}^{-q\tau^-_{s-g(s,j)}}f(s,j);\tau^-_{s-g(s,j)}<\tau_s^+ \right]}\notag \\
& \quad \quad \quad \quad \quad \quad \displaystyle{+\mathbb{E}_{(x,s,i,j)}\left[ {\rm e}^{-q\tau^+_s } V_g(s,s,\bar{J}_{\tau_s^+},\bar{J}_{\tau_s^+});\tau_s^+<\tau^-_{s-g(s,j)}    \right]  }\, .
\label{twoterms}
\end{eqnarray}
The first term in \eqref{twoterms}   is otherwise written as
\begin{equation}
f(s,j)  \sum_k \mathbb{E}_{(x,s,i,j)}[{\rm e}^{-q\tau^-_{s-g(s,j)}};\tau^-_{s-g(s,j)}<\tau_s^+, J_{\tau^-_{s-g(s,j)}  } =k   ]\, .
\label{1}
\end{equation}
The second term in \eqref{twoterms} can also be written as
\begin{equation}
 \sum_k \mathbb{E}_{(x,s,i,j)} [{\rm e}^{-q\tau^+_s} V_g(s,s,k,k);\tau_s^+<\tau^-_{s-g(s,j)},J_{\tau^+_s} =k] \, .
 \label{2}
\end{equation}
Both expectations in \eqref{1} and \eqref{2} can be identified from Proposition \ref{prop1}. More precisely,
\begin{align}
&	V_g(x,s,i,j)\notag\\
	& \displaystyle{ =f(s,j)\sum_k   [\Z^{(q)}(x-s+g(s,j))-\W^{(q)}(x-s+g(s,j))\W^{(q)}(g(s,j))^{-1}\Z^{(q)}(g(s,j)) ]_{i,k} } \nonumber\\
	& \quad \quad \quad \quad   \displaystyle{+\sum_k  [\W^{(q)}(x-s+g(s,j))\W^{(q)}(g(s,j))^{-1}]_{i,k}  V_g(s,s,k,k)    } \nonumber\\
	& \displaystyle{ =f(s,j) [\Z^{(q)}(x-s+g(s,j))1\!\!1-\W^{(q)}(x-s+g(s,j))\W^{(q)}(g(s,j))^{-1}\Z^{(q)}(g(s,j))1\!\!1 ]_{i} } \nonumber\\
	& \quad \quad \quad \quad   \displaystyle{+ \left[[\W^{(q)}(x-s+g(s,j))\W^{(q)}(g(s,j))^{-1}  V_g(s,s,\cdot,\cdot) ]  \right]_i } \; . \nonumber\\
 \label{V}
 \end{align}

With this general form of $V_g$, it is customary (cf. \cite{shep1, shep2, AK, curdin, KO, O}) to optimise over the possible choices by invoking  one condition at the point of reflection of $\overline{X} - X$ and one of two possible conditions that best describe the pasting of the value function onto the gain function.
\bigskip

We start with the pasting principle. In \cite{AAP} and then more generally in \cite{Alili}, it was noted that, for optimal stopping problems driven by L\'evy processes, the formulation of the pasting principle was dictated by the regularity of the stopping region for the underlying L\'evy process when issued on its boundary. Quite simply, if the stopping region is irregular in this respect, then a principle of continuous pasting is needed. In the case of regularity, a principle of smooth fit is needed.

\bigskip

In the current setting, we need to take account of the two different types of path variation which can occur in the ordinate among the different modulator states.
In the notation of \eqref{e:MAP F}, we will work with a principle of continuous fit for those $i\in E$ for which $\psi_i$ corresponds to a bounded variation L\'evy process, i.e., $i\in E^{\texttt{bv}}$, and a principle of smooth fit for $i\in  E^{\texttt{ubv}}$, those states of the modulator for which the ordinate moves as an unbounded variation L\'evy process. That is to say, we will insist on
\[
\lim_{x\downarrow s-g(s,j)}V_g(x, s, i,j) = f(s,j)\qquad \text{ for } i \in E^{\texttt{bv}}, s\in\mathbb{R},
\]
and
\begin{equation} \label{smooth}
\lim_{x\downarrow s-g(s,j)} \frac{\partial V_g}{\partial x} (x,s,i,j)=\frac{\partial}{\partial x}f(s, j)= 0\qquad \text{ for }i \in E^{\texttt{ubv}}s\in\mathbb{R}.
\end{equation}

Thanks to parts  (ii) and (iii) of Theorem \ref{ZWlem}, we can now easily verify that for $i\in E^{\texttt{bv}}$, $j\in E$ and $s\in\mathbb{R}$,
\begin{align}
&\lim_{x\downarrow s-g(s,j)}	V_g(x,s,i,j)	\notag\\
&
=f(s,j) [1\!\!1-\W^{(q)}(0+)\W^{(q)}(g(s,j))^{-1}\Z^{(q)}(g(s,j))1\!\!1 ]_{i}  \nonumber\\
	& \quad \quad \quad \quad   \displaystyle{+ \left[\W^{(q)}(0+)\W^{(q)}(g(s,j))^{-1}V_g(s,s,\cdot,\cdot)  \right]_i }\notag\\
	&=f(s,j) \notag\\
	&- \sum_{k\in E}[\W^{(q)}(0+)\W^{(q)}(g(s,j))^{-1}]_{i,k}\left[f(s,j)[\Z^{(q)}(g(s,j))1\!\!1 ]_k- V_g(s,s,k,k)\right]\; . \label{V0}
 \end{align}
Hence, in order to respect the continuous pasting principle for $i\in E^{\texttt{bv}}$,  providing $\texttt{card}(E^{\texttt{bv}}) >0$, it would be sufficient to  enforce the requirement that
\begin{equation}
f(s,j)[\Z^{(q)}(g(s,j))1\!\!1]_k = V_g(s,s,k,k), \qquad k\in E, s\in\mathbb{R}.
\label{enforced}
\end{equation}

\bigskip

On the other hand, if  $\texttt{card}(E^{\texttt{ubv}}) >0$, we need to be sure that for $i\in E^{\texttt{ubv}}$ that smooth pasting is possible under the assumption \eqref{enforced}.
To this end, note that, for $i\in E^{\texttt{ubv}}$,  $j\in E$ and $s\in\mathbb{R}$,
\begin{align*}
&\lim_{x\downarrow s-g(s,j)} \frac{\partial V_g}{\partial x} (x,s,i,j)\\
&=f(s,j)[\W^{(q)}(0+)(q\I -\Q)1\!\!1]_i  -f(s,j)  [\W^{(q) \prime}(0+)  \W^{(q)}(g(s,j))^{-1} \Z^{(q)}(g(s,j))1\!\!1]_{i} \\
&  +[\W^{(q) \prime}(0+) \W^{(q)}(g(s,j))^{-1}][ V_g(s,s,\cdot,\cdot)]_{i}\\
&=qf(s,j)\sum_{k\in E^{\texttt{bv}}} [\W^{(q)}(0+)]_{i,k}  \\
& - \sum_{k\in E}[\W^{(q) \prime}(0+) \W^{(q)}(g(s,j))^{-1}]_{i,k}
\left[f(s,j)[\Z^{(q)}(g(s,j))1\!\!1]_k -  V_g(s,s,k,k)\right].
\end{align*}
Thanks to Theorem \ref{ZWlem}, once assumption \eqref{enforced} is enforced, we also see that smooth pasting holds if and only if $i\in E^{\texttt{ubv}}$.

\bigskip

Now that it is clear that  \eqref{enforced} is a naturally occurring condition to satisfy the folklore of continuous and smooth pasting principles, we can substitute it into \eqref{V} and  get
\begin{equation}
V_g(x,s,i,j)=  f(s,j)  [\Z^{(q)}(x-s+g(s,j))1\!\!1]_{i}, \qquad i\in E, x\leq s.
\label{simpleV}
\end{equation}
Our heuristic reasoning has now produced a candidate value function which satisfies  \eqref{taug}, \eqref{Vg} as well as \eqref{enforced}.
\bigskip

Again referring to the historical treatment \cite{shep1, shep2, AK, curdin, KO, O}, another important feature of the Peskir--Ott heuristic for this family of optimal stopping problems (at least in the spectrally negative setting) is that a Neumann condition must hold corresponding the process of reflection of the ordinate in its running maximum.  More precisely, now working with the assumption \eqref{enforced} so that $V_g$ respects \eqref{simpleV} it would be typical to assume that
\begin{equation}
\left.\frac{\partial}{\partial s} V_g(x,s,i,i)\right|_{x \uparrow s}=0, \qquad i\in E, s\in\mathbb{R}.
\label{neumann1}
\end{equation}
A simple differentiation yields
\begin{align}
\left.\frac{\partial}{\partial s} V_g(x,s,i,i)\right|_{x \uparrow s}&=
  f^\prime(s,j)  [\Z^{(q)} (g(s,j)) 1\!\!1]_{i}  \notag\\
 &\hspace{1cm}	 +   f(s,j) (g'(s,j)-1) [\Z^{(q)\prime} (g(s,j)) 1\!\!1]_{i}\, ,
\label{neumann2}
\end{align}
for $i\in E$, providing $f(s,j)$ is continuously differentiable for $j\in E$.

Thus insisting on \eqref{neumann1} yields in \eqref{neumann2} that the unknown barrier $g(s,j)$, $s\in \mathbb{R}$ and $j\in E$, satisfies the differential equation
\begin{eqnarray}
g^\prime(s,j) &= & 1-\frac{ f^\prime(s,j)}{f(s,j)} \frac{[\Z^{(q)} (g(s,j)) 1\!\!1]_{j}}{ [ \Z^{(q)\prime} (g(s,j)) 1\!\!1]_{j}} \nonumber\\
&= &
1-\frac{ f^\prime(s,j)}{f(s,j)} \frac{[\Z^{(q)} (g(s,j)) 1\!\!1]_{j}}{ [ \W^{(q)} (g(s,j))\,(q\I-\Q)1\!\!1]_{j}} \nonumber \\ 
&= &
1-\frac{ f^\prime(s,j)}{f(s,j)} \frac{[\Z^{(q)} (g(s,j)) 1\!\!1]_{j}}{ [ q\W^{(q)} (g(s,j)) 1\!\!1]_{j}} \qquad \qquad j\in E, s\in\mathbb{R}. \label{g}
\end{eqnarray}

\bigskip

In conclusion, rehearsing the Peskir--Ott heuristic in the current context  means that we need to assume the classical assumptions of smooth and continuous pasting, depending on the modulator, which yields the simple proposed form of the optimal solution \eqref{simpleV}, as well as the classical enforcement of the Neumann condition, which  pins down the unknown optimal threshold curve $g$ in the form of the system of ODEs \eqref{g}.
As is now classical in the theory of optimal stopping, we proceed to take this informed guess for the analytic structure for the solution and verify directly the proposed solution is indeed correct; in other words, we use the method of {\it guess and verify}.

\subsection{Verification of Optimality}

In developing the proof on the basis of the guess and verify method, we
need the following Lemma (which is analogous to Lemma 11.1 in \cite{kbook}) that gives us simple criteria to verify.

\begin{lem}\label{verify}
Suppose that the conditions of Theorem \ref{main} hold. Then the pair $(\tau_g, V_g)$ attains the optimal value of the optimal stopping problem \eqref{OST}, if the following three conditions hold:

\begin{itemize}
\item[(i)] For all $x\leq s$ and $i,j\in E$,
\[
V_g(x,s,i,j) = \mathbb{E}_{(x,s,i,j)}[{\rm e}^{-q \tau_g} f(\overline{X}_{\tau_g}, \bar{J}_{\tau_g})];
\]
\item[(ii)]For all $x\leq s$ and $i,j\in E$,
\[
V_g(x,s,i,j) \geq f(s,j);
\]

\item[(iii)] The process
\[
{\rm e}^{-q t}V_g(X_t,\overline{X}_t,J_t,\bar{J}_t), \qquad t\geq 0
\]
is a right-continuous supermartingale.
\end{itemize}
\end{lem}
\begin{proof} Assumption (i) implies that
\begin{equation}
V_g(x,s,i,j) = \mathbb{E}_{(x,s,i,j)}[{\rm e}^{-q \tau_g} f(\overline{X}_{\tau_g}, \bar{J}_{\tau_g})]
\leq \sup_{\tau} \mathbb{E}_{(x,s,i,j)}[{\rm e}^{-q \tau} f(\overline{X}_{\tau}, \bar{J}_{\tau})].
\label{lowerbound}
\end{equation}
Assumption (iii), Doob's Optional Sampling Theorem and then assumption (ii) implies that
\begin{align}
V_g(x,s,i,j) &\geq  \mathbb{E}_{(x,s,i,j)}\left[{\rm e}^{-q (t\wedge \tau )}V_g(X_{t\wedge \tau },\overline{X}_{t\wedge \tau },J_{t\wedge \tau },\bar{J}_{t\wedge \tau })\right]\notag\\
&\geq  \mathbb{E}_{(x,s,i,j)}[{\rm e}^{-q (t\wedge \tau )} f(\overline{X}_{t\wedge \tau}, \bar{J}_{t\wedge \tau})],
\label{Doob}
\end{align}
for all $\tau$ in the desired class of stopping times.  Now taking limits as $t\uparrow\infty$ in \eqref{Doob} and using Fatou's Lemma, we obtain the opposite inequality to \eqref{lowerbound} and the lemma is thus proved.
\end{proof}

{\color{black}
\begin{rem}\label{necessary}\rm The ODE \eqref{gODE} need not have a unique solution and the above proof hints at why the sufficient conditions in Theorem \ref{main} will force us to single out  a specific one.   Roughly speaking, properties (i) and (ii), although stated as sufficient conditions, they  are known to be necessary conditions for the existence of an optimal solution to \eqref{OST}, in which case, inequality \eqref{Doob} indicates that
\[
\inf_gV_g(x,s,i,j)\geq  \mathbb{E}_{(x,s,i,j)}[{\rm e}^{-q \tau } f(\overline{X}_{ \tau}, \bar{J}_{ \tau})]
\]
must hold. As such, if there are many solutions to  \eqref{gODE}, we will be forced to work with the one that produces the minimisation above. This is equivalent to what Peskir \cite{maximality_principle} refers to as {\it maximality principle}.
 \end{rem}
}

In order to address criterion (iii) in our use of Lemma \ref{verify}, we need to prove  two  intermediary results.
Below, $\W_{\Phi(q)}$  denotes the matrix that plays the role of $\W^{(0)}$ under $\mathbb{P}^{\Phi(q)}$ defined in \eqref{alaGirsanov}. Also, let $F_{i,j}$ be the respective distributions of the random variables $U_{i,j}$.

\begin{lem}
\label{mmg} The process
\begin{align}
& m_{t}(X, J)\notag\\&:=  \int_{0}^{t} [ \Z^{(q)\prime}\left(X_{u-}\right) 1\!\!1]_{J_u} \sigma_{J_u} {\rm d} B_{u}
\notag\\
&+\sum_{0<u \leq t}\mathbf{1}_{(J_{u-} = J_u)}\left([( \Z^{(q)}(X_{u}) -\Z^{(q)}(X_{u-}))1\!\!1]_{J_u}
-\mathbf{1}_{(\Delta X_{u} \geq-1)}\Delta X_{u} [\Z^{(q)\prime}\left(X_{u-}\right) 1\!\!1]_{J_u} \right) \notag\\
&-\int_{0}^{t} \int_{(-\infty, 0)}\left([\Z^{(q)}\left(X_{u-}+y\right)1\!\!1]_{J_u}-[\Z^{(q)}\left(X_{u-}\right)1\!\!1]_{J_u}-y \mathbf{1}_{\{y \geq-1\}}  [\Z^{(q)\prime}(X_{u-})1\!\!1]_{J_u} \right) \nu_{J_u}({\rm d}y) {\rm d}u\notag\\
&+\sum_{0<u \leq t}\mathbf{1}_{(J_{u-} \neq J_u)}\left([\Z^{(q)}(X_{u})1\!\!1]_{J_u} -  [\Z^{(q)}(X_{u-})1\!\!1]_{J_{u-}}\right)\notag\\
&-\int_0^t \sum_{k = 1}^N\int_{(-\infty,0)}\left([\Z^{(q)}\left(X_{u-}+y\right)1\!\!1]_{k}-[\Z^{(q)}\left(X_{u-}\right)1\!\!1]_{J_{u-}}\right) q_{J_{u-},k}F_{J_{u-},k}({\rm d}y) {\rm d}u,
\label{mcalclem}
\end{align}
for $t\geq 0$, is a martingale.
\end{lem}
\begin{proof}
Note that the jump component of $m(X, J)$ centred by its compensator and thus, the part of $m(X, J)$  that is not a Brownian integral is automatically a martingale as soon as we can show that, for all $x\in\mathbb{R}$, $i\in E$ and $t\geq 0$, the $L_1$-isometry conditions
\begin{align}
&\mathbb{E}_{(x,i)}\Bigg[\int_{0}^{t} 
\int_{(-\infty, 0)}\Bigg|[\Z^{(q)}\left(X_{u-}+y\right)1\!\!1]_{J_u}-[\Z^{(q)}\left(X_{u-}\right)1\!\!1]_{J_u}\notag\\
&\hspace{7cm}-y \mathbf{1}_{\{y \geq-1\}}  [\Z^{(q)\prime}(X_{u-})1\!\!1]_{J_u} \Bigg| \nu_{J_u}({\rm d}y) {\rm d}u\Bigg]<\infty
\label{<oo1}
\end{align}
 and
\begin{align}
\mathbb{E}_{(x,i)}\Bigg[\int_0^t 
\sum_{k = 1}^N\int_{(-\infty,0)}\left|[\Z^{(q)}\left(X_{u-}+y\right)1\!\!1]_{k}-[\Z^{(q)}\left(X_{u-}\right)1\!\!1]_{J_{u-}}\right| q_{J_{u-},k}F_{J_{u-},k}({\rm d}y) {\rm d}u
\Bigg]<\infty
\label{<oo2}
\end{align}
hold.
For the second of these two verifications, note that we can use the  elementwise monotonicity of $\Z^{(q)}$ to  otherwise  write the left-hand side of \eqref{<oo2} as bounded above by
\begin{align}
&2 ||\Q||\mathbb{E}_{(x,i)}\Bigg[\int_0^t {\rm e}^{-qu}\big|{1\!\!1}^{T}\Z^{(q)}\left(X_{u}\right)1\!\!1 \big|{\rm d}u \Bigg],
\label{UB}
\end{align}
where $||Q|| = \sup_{i,j\in E, i\neq j}q_{i,j}$.
From the definition of $\Z^{(q)}$ in \eqref{Z}, we also note that
\bigskip

\[
{1\!\!1}^{T}\Z^{(q)}\left(X_{u}\right)1\!\!1 = 1 + q\int_0^{X_u}  ({1\!\!1}^{T}\W^{(q)}\left(y\right)1\!\!1
)\dd y.
\]
\bigskip

Moreover, recall from equations (16) and (19) in  \cite{kypal}, we can write
\begin{equation}
\label{Wqdecomp}
\W^{(q)}(x) ={\rm e}^{\Phi(q)x}\bDelta _{\boldsymbol{v}}\left( \Phi(q)\right)
\W_{\Phi(q)}(x)
\bDelta _{\boldsymbol{v}}\left( \Phi(q)\right)^{-1},
\end{equation}
where $\Phi(q)$ was defined in \eqref{kinverse}.
Note, we can also see from below\footnote{Note that there is a typo in the sign of the exponent  in  \cite{kypal}.} equation (16) in \cite{kypal}, on account of the fact that $((X, J),\mathbb{P}^{\Phi(q)})$ drifts to $+\infty$, there exist a family of sub-stochastic intensity matrices $(\bLambda^*_q(y), y\geq 0)$ such that
\begin{equation}
\W_{\Phi(q)}(x) = \exp\left(\int_x^\infty \bLambda^*_q (y)\dd y\right) \, \W_{\Phi(q)}(\infty) 
\qquad x\geq 0.
\label{pos}
\end{equation}
The matrix exponential on the right-hand side of \eqref{pos} is a transition semigroup (of a time-inhomogenous Markov chain), which means that its entries are all non-negative.
Then, note that the sign of the entries of $\W_{\Phi(q)}(x)$  may be positive or negative as dictated by $\W_{\Phi(q)}(\infty)$.
Recalling that the eigenvector $\boldsymbol{v}(\Phi(q))$ is element-wise strictly bounded away from 0 and $\infty$,  we can use \eqref{Wqdecomp} together with \eqref{pos} to deduce that
\[
 |{1\!\!1}^{T}\Z^{(q)}\left(X_{u}\right)1\!\!1 | \leq 1 + C_1\int_0^{X_u} {\rm e}^{\Phi(q)y}\dd y\leq C_2{\rm e}^{\Phi(q)X_u}
\]
for (unimportant) constants  $C_1, C_2>0$. Returning to \eqref{UB}, it is now easy to see that \eqref{<oo2} holds by making use of the exponential change of measure \eqref{alaGirsanov}.

\bigskip

For the sake of brevity we leave the proof of  \eqref{<oo1} as an exercise for the reader, noting that its proof goes along similar lines.

\bigskip
 In order to justify that  $ \int_{0}^{t} [ \Z^{(q)\prime}\left(X_{u-}\right) 1\!\!1]_{J_u} \sigma_{J_u}{\rm d} B_{u}^{J_u}$, $t\geq 0$, is a martingale, it suffices to check that its  mean quadratic variation is finite; that is to say, the necessary $L_2$-isometry condition holds. Using \eqref{Wqdecomp} and \eqref{psi2}, the aforementioned is verified via
\begin{align*}
&\mathbb{E}_{(x,i)}\left[\int_0^t 
[ \Z^{(q)\prime}\left(X_{u-}\right) 1\!\!1]_{J_u}^2 \sigma^2_{J_u}\dd u\right]=
\mathbb{E}_{(x,i)}\left[\int_0^t 
[ q\W^{(q)}\left(X_{u-}\right) 1\!\!1]_{J_u}^2 \sigma^2_{J_u}\dd u\right]\\
&\leq q^2\bar\sigma^2\frac{\bar{v}}{\underline{v}} ||\W_{\Phi(q)}||^{2}
\, \mathbb{E}_{(x,i)}\left[\int_0^t
\mathbf{1}_{(X_t\geq 0)}{\rm e}^{2\Phi(q) X_u} \dd u\right]\\
&\leq q^2\bar\sigma^2\frac{\bar{v}}{\underline{v}} ||\W_{\Phi(q)}||^2
\int_0^t \mathbb{E}_{(x,i)}
\left[
{\rm e}^{2\Phi(q) X_u}
\right]\dd u\\
&=q^2\bar\sigma^2\frac{\bar{v}}{\underline{v}} ||\W_{\Phi(q)}||^2
\int_0^t 
[{\rm e}^{\bPsi(2\Phi(q))u}1\!\!1]_i \dd u<\infty,
\end{align*}
  where  $||\W_{\Phi(q)}|| = \sup_{i,j\in E}|[\W_{\Phi(q)}(x)]_{i,j}|$,  $\bar\sigma^2 = \max_{i \in E}\sigma_i^2$, $\bar{v} = \max_{i\in E}v_i(\Phi(q))$ and $\underline{v} = \min_{j\in E}v_j(\Phi(q))$, and the final inequality follows again from the martingale on the right-hand side of \eqref{alaGirsanov}. To see why $ \sup_{i,j\in E}|\W_{\Phi(q)}(x)|_{i,j}<\infty$, we recall from the change of measure \eqref{alaGirsanov} that the process $(\mathbb{P}^{\Phi(q)}, X)$ drifts to $+\infty$. Hence, recalling \eqref{id1},
\begin{align*}
0&<\mathbb{P}^{\Phi(q)}_{(x,i)}(\tau_{0}^{-}=\infty,J_{\tau_{a}^{+}}=j) \\
&\hspace{0.5cm}= \lim_{a\to\infty}\mathbb{P}^{\Phi(q)}_{(x,i)}(\tau_{a}^{+}<\tau_{0}^{-},J_{\tau_{a}^{+}}=j)\\
&\hspace{1cm}=\lim_{a\to\infty}[\W_{\Phi(q)}(x)\W_{\Phi(q)}(a)^{-1}]_{i,j}\\
&\hspace{1.5cm}\leq 1,
\end{align*}
which forces $ \sup_{i,j\in E}|\W_{\Phi(q)}(x)|_{i,j}<\infty$.
\end{proof}
We need to define some infinitesimal generators of some of the processes that make up the underlying MAP. To this end, in the case that $i\in E^{\texttt{bv}}$, for convenience, we will write
\[
\psi_{i}(\lambda) = a_i\lambda + \int_{(-\infty,0)}({\rm e}^{\lambda x}-1)\nu_i({\rm d} x), \qquad \lambda\geq 0,
\]
where $a_i>0$ and $\int_{(-\infty,0)}(|x|\wedge 1)\nu_i({\rm d} x)<\infty$, which is always possible for spectrally negative L\'evy processes, cf. equation (8.3) in \cite{kbook}. When $i\in E^{\texttt{ubv}}$, we will instead need to identify the Laplace exponent
\[
\psi_i(\lambda) = a_i\lambda +\frac{1}{2}\sigma^2_i \lambda^2 + \int_{(-\infty,0)} ({\rm e}^{\lambda x} - 1 - x\mathbf{1}_{(|x|<1)})\nu_i({\rm d}x), \qquad \lambda \geq 0,
\]
where $a_i\in \mathbb{R}$, $\sigma_i^2\geq 0$ and $\int_{(-\infty,0)}(|x|^2\wedge 1)\nu_i({\rm d}x)<\infty$. Accordingly, we identify the two associated infinitesimal generators. The first is
\[
\mathcal{A}^if(x)=a_{i} f^{\prime}(x+)+\int_{(-\infty,0)}\left(f(x+y)-f(x)\right) \nu_{i}(\mathrm{d} y),\qquad i\in E^{\texttt{bv}}, x\in\mathbb{R},
\]
for $f\in C^{1,+}_0(\mathbb{R})$, the set of functions which have a right-continuous derivative and which vanish at $-\infty$. The second is
\[
\mathcal{A}^if(x)=a_{i} f^{\prime}(x)+\frac{\sigma_{i}^2}{2} f^{\prime \prime}(x+)+\int_{(-\infty,0)}\left(f(x+y)-f(x)- y f^{\prime}(x) \mathbf{1}_{(|x|<1)}\right)\nu_i({\rm d}x),
\]
for $i\in E^{\texttt{ubv}}$, $x\in\mathbb{R}$ and $f\in C^{2,+}_0(\mathbb{R})$, the space of functions which are continuously differentiable with right-continuous second derivative and which vanish at $-\infty$ ($f(x)\rightarrow 0$  as $x\rightarrow -\infty$).

\bigskip

We also need to introduce the generator that codes the rate at which the Markov chain jumps and causes an additional discontinuity in the ordinate. Define

\[
\mathcal{B} f(x, i)=\sum_{k=1}^{N} q_{i k} \int_{(-\infty,0)}(f(x+y, k)-f(x, i))  F_{i,k}(\mathrm{d} y), \qquad i\in E, x\in\mathbb{R},
\]
where, for $i,j\in E$ we recall that $q_{i,k}$ are the entries of the matrix $\Q$ and $F_{i,j}$ are the respective distributions of the random variables $U_{i,j}$.
\bigskip

\begin{prop}\label{lem} For all  $i\in E$, we have
	\[
	H_i(x): = \mathcal{A}^i([\Z^{(q)}(\cdot)1\!\!1]_{i}) (x) \,  +\mathcal{B}([\Z^{(q)}(\cdot)1\!\!1]_{\cdot}) (x,i)-q[\Z^{(q)}(x)1\!\!1]_{i} = 0
	\]	
for $x\geq0$, and  $H_i(x)<0$ for $x<0$.
\end{prop}
\begin{proof}  We start by proving the claim that
\begin{equation}
u(x,i): = \mathbb{E}_{(x,i)}\left[ {\rm e}^{-q\tau^-_0} [\W^{(q)}(X_{\tau^-_0})\W^{(q)}(a)^{-1}1\!\!1]_{J_{\tau^-_0}}; X_{\tau^-_0} = 0 , \tau^-_0<\tau^+_a\right] = 0,
\label{W0}
\end{equation}
for all $x\geq 0$. To this end, we break the expectation on the right-hand side of \eqref{W0} according to the exhaustive disjoint union of events $\bigcup_{n\geq 0}\{\tau^-_0\in [\sigma_n,\sigma_{n+1})\}$, where $\sigma_0 = 0$ and $\sigma_n$ is the $n$-th jump time of the modulator $J$, for $n\geq 1$. This, together with the Markov property gives us
 \begin{align*}
& u(x,i) \\
&= \mathbb{E}_{(x,i)}\left[{\rm e}^{-q\tau^-_0} [\W^{(q)}(0+)\W^{(q)}(a)^{-1}1\!\!1]_{J_{\tau^-_0}}; X_{\tau^-_0} = 0 , \tau^-_0<\tau^+_a\wedge \sigma_1\right]\\
 &+ \sum_{n\geq 1}\mathbb{E}_{(x,i)}\left[{\rm e}^{-q \sigma_n}
 \mathbf{1}_{(\sigma_n < \tau^-_0\wedge \tau^+_a)}\mathbb{E}_{(X_{\sigma_n},J_{\sigma_n}) }\left[ [\W^{(q)}(0+)\W^{(q)}(a)^{-1}1\!\!1]_{J_{\tau^-_0}}; X_{\tau^-_0} = 0 , \tau^-_0<\tau^+_a\wedge \sigma_1\right]\right]\\
 &= \mathbb{E}_{(x,i)}\left[{\rm e}^{-q\tau^-_0} [\W^{(q)}(0+)\W^{(q)}(a)^{-1}1\!\!1]_{i}; X_{\tau^-_0} = 0 ,  \tau^-_0<\tau^+_a\wedge \sigma_1\right]\\
 &+ \sum_{n\geq 1}\mathbb{E}_{(x,i)}\left[{\rm e}^{-q \sigma_n}
 \mathbf{1}_{(\sigma_n < \tau^-_0\wedge \tau^+_a)}\mathbb{E}_{(X_{\sigma_n}, j) }\left[ [\W^{(q)}(0+)\W^{(q)}(a)^{-1}1\!\!1]_{j}; X_{\tau^-_0} = 0 , \tau^-_0<\tau^+_a\wedge \sigma_1\right]_{J_{\sigma_n} =j}\right],
 \end{align*}
 where we have used the fact that, if $J_0 =i$, then  $J_{\tau^-_0} = i$ on the event $\tau^-_0<\sigma_1$.

\bigskip

 Next note that for an expectation  of the form
 \begin{equation}
 \mathbb{E}_{(x, j) }\left[ [\W^{(q)}(0+)\W^{(q)}(a)^{-1}1\!\!1]_{j}; X_{\tau^-_0} = 0 , \tau^-_0<\tau^+_a\wedge \sigma_1\right],
\label{expectation}
 \end{equation}
 if $j\in E^{\texttt{ubv}}$, then $[\W^{(q)}(0+)\W^{(q)}(a)^{-1}1\!\!1]_{j}=0$ thanks to Theorem \ref{ZWlem}, or, otherwise, if $j\in E^{\texttt{bv}}$, then $\{X_{\tau^-_0} = 0, \, \tau^-_0< \sigma_1\}$ is almost surely the empty set. Either way, the expectation \eqref{expectation} is zero and, hence,  $u(x,i) = 0$ for all $x\in \mathbb{R}$ and $i\in E$, as claimed in \eqref{W0}.
 \bigskip

 With \eqref{W0} in hand, we now note that either
 \[
 [\W^{(q)}(X_{\tau^-_0})\W^{(q)}(a)^{-1}1\!\!1]_{J_{\tau^-_0}} = 0\quad \text{  or  }\quad\mathbf{1}_{(X_{\tau^-_0 }= 0)}=0
 \]
 almost surely on the event $\{\tau^-_0< \tau^+_a\}$. Noting that $[\W^{(q)}(X_{\tau^+_a})\W^{(q)}(a)^{-1}1\!\!1]_{J_{\tau^+_a}} = 1$ almost surely, it follows that, with    $\tau_{0,a} : = \tau^+_a \wedge \tau^-_0$,
 \begin{equation}
 {\rm e}^{-q \tau^+_a}   1_{\{\tau_a^+<\tau_0^-\}}    = {\rm e}^{-q \tau_{0,a}} [\W^{(q)}(X_{\tau_{0,a}})\W^{(q)}(a)^{-1} 1\!\!1 ]_{J_{\tau_{0,a}}}.
 \end{equation}
 Taking expectations, this gives us, for $x\in\mathbb{R}$, $i\in E$.
 \[
\mathbb{E}_{(x,i)} \left[{\rm e}^{-q \tau_a^+}   1_{\{\tau_a^+<\tau_0^-\}}  \right ]= \mathbb{E}_{(x,i)}  \left[{\rm e}^{-q \tau_{0,a}} [\W^{(q)}(X_{\tau_{0,a}})\W^{(q)}(a)^{-1} 1\!\!1]_{J_{\tau_{0,a}} }\right].
\]
 This and Markov property imply
\begin{align}
& \mathbb{E}_{(x,i)}\left[{\rm e}^{-q \tau_{0,a}} [\W^{(q)}\left(X_{\tau_{0,a}}  \right)\W^{(q)}(a)^{-1} 1\!\!1]_{J_{\tau_{0,a}}}  | \mathcal{G}_{t}\right] \notag\\
& = \mathbf{1}_{(t \leq \tau_{0,a})} {\rm e}^{-q t} \mathbb{E}_{X_{t}, J_t}\left[{\rm e}^{-q \tau_{0,a}} [\W^{(q)}\left(X_{\tau_{0,a}}\right)\W^{(q)}(a)^{-1}1\!\!1]_{J_{\tau_{0,a}}}\right] \\
&\quad \quad +\mathbf{1}_{(t>\tau_{0,a})} {\rm e}^{-q \tau_{0,a}} [\W^{(q)}\left(X_{\tau_{0,a}}\right)\W^{(q)}(a)^{-1}1\!\!1] _{J_{\tau_{0,a}}}\notag\\
&=\mathbf{1}_{(t \leq \tau_{0,a})} {\rm e}^{-q t} [\W^{(q)}\left(X_{t}\right)\W^{(q)}(a)^{-1}1\!\!1]_{J_t} +\mathbf{1}_{(t>\tau_{0,a})} {\rm e}^{-q \tau} [\W^{(q)}\left(X_{\tau_{0,a}}\right)\W^{(q)}(a)^{-1}1\!\!1]_{J_{\tau_{0,a}}} \\ &={\rm e}^{-q({t \wedge \tau_{0,a}})} [\W^{(q)}\left(X_{{t \wedge \tau_{0,a}}}\right)\W^{(q)}(a)^{-1}1\!\!1]_{J_{{t \wedge \tau_{0,a}}}}\; .
\label{mgcomputation}
\end{align}
In other words, we have that
\begin{equation} \label{mart0}
{\rm e}^{-q({t \wedge \tau_{0,a}})} [\W^{(q)}\left(X_{{t \wedge \tau_{0,a}}}\right)\W^{(q)}(a)^{-1}1\!\!1]_{J_{{t \wedge \tau_{0,a}}}}, \qquad t\geq 0,
\end{equation}
is a martingale.
In a similar spirit, noting from Theorem \ref{ZWlem} that $\Z^{(q)}(x) = \I$ for $x\leq 0$, we can deduce   with the help of \eqref{id2} in Proposition \ref{prop1}, that
\begin{equation}
{e}^{-q({t \wedge \tau_{0,a}} )}[\Z^{(q)}(X_{{t \wedge \tau_{0,a}}})1\!\!1]_{J_{{t \wedge \tau_{0,a}}}} -{e}^{-q({t \wedge \tau_{0,a}} )} [ \W^{(q)}(X_{{t \wedge \tau_{0,a}}})\W^{(q)}(a)^{-1}\Z^{(q)}(a)1\!\!1]_{J_{{t \wedge \tau_{0,a}}}}
\end{equation}
is a martingale as well. As linear combinations of martingales are still martingales, we thus have that \begin{equation} \label{mart1}
{e}^{-q({t \wedge \tau_{0,a}} )}[\Z^{(q)}(X_{{t \wedge \tau_{0,a}}})1\!\!1]_{J_{{t \wedge \tau_{0,a}}}}
\end{equation}
is also a martingale.

\bigskip

Recall from Theorem \ref{ZWlem} (iv) that $\Z^{(q)}$ is continuously differentiable, except possibly at $0$ and otherwise is almost everywhere twice differentiable with a right-continuous second derivative on $\mathbb{R}$ (which is thus locally bounded). Along the intervals of time between  jumps of the MAP modulator, i.e. $[\sigma_n, \sigma_{n+1})$, for $n\geq 0$, we can  apply piecewise the  version of It\^o's formula in Theorem 3 of  \cite{EK}, together with the conclusion of Lemma 7 (ii) in the same paper (which permits us to write It\^o's formula in the same way as usual despite the slightly weaker smoothness assumptions), and get, on $\{t\leq \tau_{0,a}\}$,

\begin{eqnarray}
	{\rm d} ({\rm e}^{-qt}[\Z^{(q)}(X_{t})1\!\!1]_{J_{t}})&=
	{\rm e}^{-qt} \mathcal{A}^{J_{t-}}([\Z^{(q)}(\cdot)1\!\!1]_{J_{t-}}) (X_{t-}) \, {\dd t} +{\rm e}^{-qt}\mathcal{B}([\Z^{(q)}(\cdot)1\!\!1]_{\cdot}) (X_{t-},J_{t-})\, {\dd t} \notag\\
	& -q {\rm e}^{-qt}[\Z^{(q)}(X_{t-})1\!\!1]_{J_{t-}} \, {\dd t} + {\rm e}^{-qt}\dd  m_t(X, J),\qquad t\geq 0.
	\label{nodrift}
\end{eqnarray}
where, from Lemma \ref{mmg} $(m_t(X, J) , t\geq 0)$ is an $\mathbb{R}$-valued  martingale.
%
%

\bigskip

Since both \eqref{mart1} and $(m_{t \wedge \tau_{0,a} }(X, J), t\geq 0)$ are  martingales,  the drift term in \eqref{nodrift} must be zero. Sampling the two aforesaid martingales at the stopping time $\sigma_1\wedge \tau_{0,a}$ (recall $\sigma_1$ is the first jump time of $J$) gives us,  for all $i\in E$, $0<x<a$ and $t\geq 0$,
\begin{align}
0&=\mathbb{E}_{(x,i)}\bigg[
\int_0^{\sigma_1 \wedge \tau_{0,a}}
	{\rm e}^{-qs} \bigg(\mathcal{A}^{J_{s-}}([\Z^{(q)}(\cdot)1\!\!1]_{J_{s-}}) (X_{s-})  \\
	&\hspace{5.5cm}+ \mathcal{B}([\Z^{(q)}(\cdot)1\!\!1]_{\cdot}) (X_{s-},J_{s-})
	 -q [\Z^{(q)}(X_{s-})1\!\!1]_{J_{s-}} \bigg) {\dd s}
\bigg]\notag\\
&=\mathbf{E}^{i}_{x}\bigg[
\int_0^{ \tau^i_{0,a}}
	{\rm e}^{-(q+|q_{i,i}|)s} H_i(X^i_{s-}) {\dd s}
\bigg],
\label{=0}
\end{align}
where $(X^i_t, t\geq 0)$, with probabilities $(\mathbf{P}^i_x, x\in \mathbb{R})$, is an independent copy of the spectrally negative L\'evy process with Laplace exponent $\psi_i$ and $\tau^i_{0,a} = \inf\{t>0: X^i_t\not\in[0,a]\}$  from which the statement of the theorem follows for $x>0$. From Section 8.4 of \cite{kbook}, we know that  the $(q+|q_{i,i}|)$-potential measure  of $X^i$  killed on exiting $[0,a]$ has a density with respect to Lebesgue measure, say $u^{(q+|q_{i,i}|)}_i(a, x,y)$, which is continuous in $q\geq 0$, $x,y\geq 0$, and  strictly positive.  The identity \eqref{=0} thus reads
\[
\int_0^a H_i(y)u^{(q+|q_{i,i}|)}_i(a, x,y)\,\dd y = 0, \qquad x\in[0,a], a>0, q>0,i\in E,
\]
 and standard arguments allow us to deduce that $H_i (x) = 0$ for Lebesgue almost every $x>0$. It is easy to verify from its definition that $H_i(x)$ is right continuous and hence  $H_i (x) = 0$ for $x\geq 0$.
\bigskip

To deal with the case $x<0$, it suffices to note that $\Z^{(q)}(x) = \I$ and one can verify by hand that
\[
	\mathcal{A}^i([\Z^{(q)}(\cdot)1\!\!1]_{i}) (x) \,  +\mathcal{B}([\Z^{(q)}(\cdot)1\!\!1]_{\cdot}) (x,i)-q[\Z^{(q)}(x)1\!\!1]_{i}  = -q<  0,
	\]
	as required. \end{proof}

 \bigskip

 \begin{proof}[Proof (of Theorem \ref{main})]
 The essence of our proof is to verify the three conditions of Lemma \ref{verify}, by appealing to the assumptions of Theorem \ref{main}.
 As noted below \eqref{simpleV}, by guessing the strategy $\tau_g$ and insisting on the smooth and continuous pasting properties as described in Section \ref{heuristic}, part (i) of Lemma \ref{verify} is verified for $s\geq s^*$.

%
%
 \bigskip

To verify part (ii) of Lemma \ref{verify}, given that
 \begin{align*}
 V_g(x,s,i,j)&=  f(s,j)  [\Z^{(q)}(x-s+g(s,j))1\!\!1]_{i} \\
 &=  f(s,j)\left[
1 + \int_0^{x-s+g(s,j)} [\W^{(q)}(y) (qI - \Q)1\!\!1]_i\,\dd y
 \right]\\
 &=  f(s,j)\left[
1 + q \int_0^{x-s+g(s,j)} [\W^{(q)}(y) 1\!\!1]_i\,\dd y
 \right]
  \end{align*}
  for $x\leq s$, $i,j\in E$, since $\W^{(q)}(y) = 0$ for $y\leq 0$, it is easy to see that, for  $s-x\geq g(s,j)$,
  \[
   V_g(x,s,i,j)= f(s,j).
  \]
  Next note that, for  $s-x\leq g(s,j)$, we have
  \[
   V_g(x,s,i,j)\ge f(s,j).
  \]
 as $[\Z^{(q)}(y) 1\!\!1]_j \ge 1$ under the assumption that $g(s,j)\le a(j)$ for all $j\in E$, where $a(j)\in (0,\infty]$ by Corollary \ref{rem2}.
  \bigskip

  Finally for part (iii) of Lemma \ref{verify}, we need to apply again an appropriate form of It\^o's formula.
%
%
%
%
To this end, note that
\[V_{g}(X_t,\overline{X}_t,J_t,\bar{J}_t)=  f(\overline{X}_t,\bar{J}_t) [\Z^{(q)}(X_{t}-\overline{X}_t+g(\overline{X}_{t},\bar{J}_t))1\!\!1]_{J_t}, \qquad t\geq 0.
\]
Let
\begin{equation} \label{Y}
Y_{t}=X_{t}-\overline{X}_t+g(\overline{X}_{t},\bar{J}_t), \qquad t\geq 0.
\end{equation}

Mixing the integration by parts formula with the earlier indicated calculus from \cite{EK}, we have, for $t\geq 0$,
\begin{align*}
	 \dd[{\rm e}^{-qt}V_{g}(X_t,\overline{X}_t,J_t,\bar{J}_t)]&={\rm e}^{-qt}f'(\overline{X}_{t-},\bar{J}_{t-})[\Z^{(q)}(Y_{t-})1\!\!1]_{J_{t-}}\,\dd\overline{X}_t\\
	 &
	+{\rm e}^{-qt}f(\overline{X}_t,\bar{J}_t)\, \dd([\Z^{(q)}(Y_{t})1\!\!1]_{J_t})
	-q{\rm e}^{-qt}f(\overline{X}_t,\bar{J}_t)[\Z^{(q)}(Y_{t})1\!\!1]_{J_t}\,\dd t.
	\end{align*}
	
Using that $Y_{t}=g(\overline{X}_{t},\bar{J}_t)$ and $J_t = \bar{J}_t$ for all $t$ in the support of $\dd \overline{X}_t$, together  with the help of \eqref{nodrift} and \eqref{mcalclem}, we can develop the calculus further to get
\begin{align}
& \dd[{\rm e}^{-qt}V_{g}(X_t,\overline{X}_t,J_t,\bar{J}_t)]\notag\\
&={\rm e}^{-qt}
 \Big(f'(\overline{X}_{t-},\bar{J}_{t-})[\Z^{(q)}(g(\overline{X}_{t-},\bar{J}_{t-}))1\!\!1]_{\bar{J}_{t-}} \notag\\
 &\hspace{2cm}+f(\overline{X}_{t-},\bar{J}_{t-})(g'(\overline{X}_{t-},\bar{J}_{t-})-1)
 [\Z^{(q)\prime}(g(\overline{X}_{t-},\bar{J}_{t-}))1\!\!1]_{\bar{J}_{t-}}
 \Big)\dd\overline{X}_t\notag\\
	 &
	+{\rm e}^{-qt}f(\overline{X}_t,\bar{J}_t)\,\left(
	 \mathcal{A}^{J_{t-}}([\Z^{(q)}(\cdot)1\!\!1]_{J_{t-}}) (Y_{t-})  +\mathcal{B}([\Z^{(q)}(\cdot)1\!\!1]_{\cdot}) (Y_{t-},J_{t-})
	-q[\Z^{(q)}(Y_{t})1\!\!1]_{J_t} \right) {\dd t}\notag \\
	 &+{\rm e}^{-qt}f(\overline{X}_t,\bar{J}_t) \,\dd  m_t(Y),\qquad t\geq 0.
\label{bits}
\end{align}
For the integral with respect to $\dd \overline{X}_t$ in \eqref{bits}, the assumption that $g$ solves \eqref{gODE} implies that its integrand is zero. For the integral with respect to $\dd t$ in \eqref{bits}, noting from  the piecewise construction of the MAP from spectrally negative L\'evy processes that $X$ spends zero Lebesgue time at $0$, we can use Proposition \ref{lem} to deduce that its integrand is strictly negative when $Y_t<0$.
\bigskip

We are thus left with
\begin{equation}
\dd[{\rm e}^{-qt}V_{g}(X_t,\overline{X}_t,J_t,\bar{J}_t)] =  {\rm e}^{-qt}f(\overline{X}_t,\bar{J}_t)\, \dd  m_t(Y) - q\mathbf{1}_{(Y_t<0)}  \dd t,
\label{smg}
\end{equation}
for $t\geq 0$, $s\geq s^*$. We know, however that $(m_t(Y),t\geq 0)$ has martingale increments and thus
$\int_0^t{\rm e}^{-qs}f(\overline{X}_s,\bar{J}_s)\, \dd  m_s(Y)$ is a local martingale. To verify it is a martingale, we leave the bulk of the details to the reader, however, in the spirit of the verification of \eqref{<oo1} and \eqref{<oo2}, it suffices to check that the necessary $L_1$- and $L_2$-isometries hold, for which it is sufficient that, for all $t\geq0$,
\[
\mathbb{E}_{(x,i)}\left[\int_0^t {\rm e}^{-qu}f(\overline{X}_u,\bar{J}_u)  {\rm e}^{\Phi(q)X_u}\dd u\right]+\mathbb{E}_{(x,i)}\left[\int_0^t {\rm e}^{-2qu}f(\overline{X}_u,\bar{J}_u)^2  {\rm e}^{2\Phi(q)X_u}\dd u\right]<\infty.
\]
This is automatically satisfied once \eqref{conditions} is fulfilled.
\bigskip

The supermartingale property in part (iii) of Lemma \ref{verify} is thus satisfied. Note in particular,
 \begin{equation}
\mathbb{E}_{(x,s,i,j)}\left[{\rm e}^{-q t}V_g(X_t,\overline{X}_t,J_t,\bar{J}_t)\right]\leq V_g(x,s,i,j), \qquad t\geq 0, x\leq s, i,j\in E.
\label{superh}
\end{equation}

Inequality \eqref{superh}, together with the Markov property, is sufficient to deduce the required supermartingale property required in (iii) of Lemma \ref{verify}. The proof is now complete.
\end{proof}

\begin{rem}\label{dontneedgdiffeq}\rm
If we look closer at the above proof, we note that the requirement that $g$ solves \eqref{gODE} emerges in the need for \eqref{bits} to be a supermartingale. It is noticeable in this respect that a weaker condition for the latter is simply that
\begin{equation}
g'(s,j)
 \leq 1-
 \frac{
f'(s,j)[\Z^{(q)}(g(s,j))1\!\!1]_{j}
 }
 {
 f(s,\bar{J}_t)[\Z^{(q)\prime}(g(s,j))1\!\!1]_{j}
  }, \qquad j\in E, s\in\mathbb{R}.
\label{gODE2}
\end{equation}
\end{rem}

\section{Shepp-Shiryaev optimal stopping problem}

In \cite{avra}, pricing of the  Russian option is considered for the spectrally negative  L\'{e}vy case, where it had earlier been introduced in the Black-Scholes setting in \cite{shep1, shep2}. Setting aside the financial connection, let us investigate here the natural analogue of this problem in the MAP setting. This is tantamount to considering the gain function in \eqref{OST} taking the form
\[
f(s,j) = {\rm e}^s h_j, \qquad s\in\mathbb{R}, j\in E,
\]
where $h_j \in\mathbb{R}_+ $.
Theorem \ref{main} guides us to checking a number of conditions, the principal one needing us to find a solution to
\eqref{gODE}. In the current setting \eqref{gODE}   takes the form
\begin{equation}   \label{dif}
g^\prime(s,j) =  
  1- \frac{  [\Z^{(q)} (g(s,j))  1\!\!1]_j}{ [q\W^{(q)}(g(s,j))1\!\!1]_j }, \qquad s\in\mathbb{R}, j\in E.
\end{equation}
Note that \eqref{dif} presents us with $N$ independent differential  equations, rather than a more complicated coupled system. Indeed, this is true in general for \eqref{gODE}.

\bigskip

Taking inspiration from \cite{avra}, which solves the analogue of  \eqref{dif} for spectrally negative L\'evy processes, we look for solutions taking the form $g(s,j)=c_j$ for each $j$.
 In theory, \eqref{gODE} suggests  we  look for roots of the equation
	\[
	  [\Z^{(q)} (x) 1\!\!1]_{j} -q [ \W^{(q)} (x)1\!\!1]_{j} = 0,
	\]

for $x\geq 0$.  However, noting from \cite{avra} that roots of the analogous equation (written in that context with the scale function $Z^{(q)}$) don't always exist, we look instead towards the suggested alternative condition to Theorem \ref{main} (i) alluded to in Remark \ref{dontneedgdiffeq}. Bearing in mind the observation in Remark \ref{necessary}, we seek a choice of constant values $g(s, j) = c_j$, $j\in E$, that satisfy \eqref{gODE2}, which is equivalent to seeking minimal solutions to
\[
u_j(x):= [\Z^{(q)}   1\!\!1]_j(x)-q[\W^{(q) }   1\!\!1]_j(x) \leq 0, \qquad x\geq 0.
\]


In summary, we will prove the following result. Recall the definition
\[
a(j)= \inf\{x> 0:  [\Z^{(q)}(x)1\!\!1]_j \le 1\}
\]
from Theorem \ref{main}.

\begin{thm}\label{russian-alternative} Assume that $q>0\vee \kappa(1)$,  $
	f(s,j) = {\rm e}^s h_j>0$, $s\in\mathbb{R}$, $j\in E$, and let
	\[
	c_{j} = \inf\{x\geq 0: u_j(x)\leq 0\}.
	\]
\begin{itemize}
    \item[(i)] When $[\W^{(q)}1\!\!1]_j(0+)\in[q^{-1},\infty)$ for $j\in E$, then $c_j =0$.
    \item[(ii)] When $[\W^{(q)}1\!\!1]_j(0+)\in [0, q^{-1})$   for $j\in E$,
assume that $\lim_{x\rightarrow \infty} [\W^{(q)}   1\!\!1]_j(x)=+\infty$. Then, $c_j\in(0,\infty)$.
\end{itemize}

If, moreover $c_j\le a(j)$ for all $j\in E$, then the  solution of the optimal stopping problem \eqref{OST} is given by
	\[
	V(x,s,i,j) = h_j{\rm e}^s [\Z^{(q)}(x-s+c_{j}) 1\!\!1]_{i},
	\qquad  x\leq s, i,j\in E,
	\]
	and the optimal strategy is given by
	\[
	\tau_{c}=\inf\{t\ge 0:\overline{X}_{t}-X_t\ge c_{\bar{J}_t}\}.
	\]
\end{thm}

In order to prove Theorem \ref{russian-alternative}, we need some intermediary results, which we deal with first.

\begin{lem}\label{boringconditions}
The condition  \eqref{conditions} holds for all  $q>0$.
\end{lem}
\begin{proof}
%
%
For \eqref{conditions} it is sufficient to show that
\begin{equation}
\int_0^t \mathbb{E}_{(x,i)}\left[ {\rm e}^{2\Phi(q)X_u + 2\overline{X}_u}\right]\dd u<\infty,
\label{conditions-sufficient}
\end{equation}
for all $t\geq0.$ To this end, noting that $\overline{X}$ is increasing,  it would thus be sufficient to show that $ \mathbb{E}_{(x,i)}\left[ {\rm e}^{2(1+\Phi(q))\overline{X}_t}\right]<\infty$.

\bigskip

In order to show the latter, we claim that, for  $p>0$ sufficiently large, we have
\begin{equation}
 \mathbb{E}_{(x,i)}\left[ {\rm e}^{2(1+\Phi(q))\overline{X}_{\mathbf{e}_p}}\right]<\infty.
 \label{sufficienttoprove}
\end{equation}
where $\mathbf{e}_p$ is an independent and exponentially distributed random variable with rate $p$.
Note from \eqref{id0}
\begin{align*}
\mathbb{P}_{(x,i)} (\overline{X}_{\mathbf{e}_p}>t)&= \mathbb{P}_{(x,i)} (\tau^+_t< {\mathbf{e}_p})\\
&=
 \mathbb{E}_{(0,i)} [{\rm e}^{-p\tau^+_{t-x}}] \\
 &=[\bDelta _{\boldsymbol{v}}\left( \Phi(p) \right)
{\rm e}^{(t-x)( \bLambda(p)-\Phi(p) \I )}
\bDelta _{\boldsymbol{v}}\left(  \Phi(p)  \right)^{-1} 1\!\!1]_{i} .
\end{align*}
Noting that $\Phi$ is an increasing function, we can choose $p$ sufficiently large such that $2(1+\Phi(q))<\Phi(p)$, in which case, \eqref{sufficienttoprove} holds.  With \eqref{sufficienttoprove} in hand, writing it in the form
\[
p\int_0^\infty {\rm e}^{-p t} \mathbb{E}_{(x,i)}\left[ {\rm e}^{2(1+\Phi(q))\overline{X}_{t} }\right]\dd t<\infty,
\]
we are led to the conclusion that $\mathbb{E}_{(x,i)}[ {\rm e}^{2(1+\Phi(q))\overline{X}_{t} }]<\infty$ for Lebesgue almost every $t>0$. The finiteness of these expectations can be deduced for all $t>0$ by observing that they form a monotone increasing sequence in $t$.
\end{proof}

The following Proposition resembles the L\'evy process counterpart in Lemma 1 of \cite{avra}. It is also related to its matrix analogue given in Lemma 1 in \cite{JI}. 
\begin{prop} \label{prop3}
For $q> 0$, we have
\[
\lim_{x\to\infty}\frac{[\Z^{(q)}   1\!\!1]_j(x)}{[\W^{(q) }   1\!\!1]_j(x)} = \frac{q}{\Phi(q)}.
\]
\end{prop}
\begin{proof}
From \eqref{Z}, recall that
\[
[\Z^{(q)}   1\!\!1]_j(x) = 1 + q\int_0^x [\W^{(q)}1\!\!1]_j(y)\dd y, \qquad x\geq 0.
\]

By \eqref{Wqdecomp}  and integration by parts, we get
\begin{align*}
[\Z^{(q)}   1\!\!1]_j(x) & = 1+\frac{q}{\Phi(q)}([\W^{(q)}1\!\!1]_j(x) - [\W^{(q)}1\!\!1]_j(0+)) \\
&\quad -  \frac{q}{\Phi(q)}\int_0^x{\rm e}^{\Phi(q)y}  [\bDelta _{\boldsymbol{v}}\left( \Phi(q)\right)
\W'_{\Phi(q)} (y+)
\bDelta _{\boldsymbol{v}}\left( \Phi(q)\right)^{-1}1\!\!1]_j\,\dd y\: .
\end{align*}
Then, the result follows once we show that
\begin{equation}  \label{ratio}
\lim_{x\to\infty}\frac{\int_0^x{\rm e}^{\Phi(q)(y-x)}  [\bDelta _{\boldsymbol{v}}\left( \Phi(q)\right)
\W'_{\Phi(q)} (y+)
\bDelta _{\boldsymbol{v}}\left( \Phi(q)\right)^{-1}1\!\!1]_j\dd y} { [\bDelta _{\boldsymbol{v}}\left( \Phi(q)\right)
\W_{\Phi(q)} (x)
\bDelta _{\boldsymbol{v}}\left( \Phi(q)\right)^{-1}1\!\!1]_j     } = 0\, ,
\end{equation}
where \eqref{Wqdecomp} is used  for $[\W^{(q)}1\!\!1]_j$. Now, in view of \eqref{pos},  we have
\[
\lim_{x\to\infty} || \W_{\Phi(q)}(x) ||<\infty
\]
and hence  the denominator in \eqref{ratio} has a finite limit as $x\rightarrow \infty$ as well as
\begin{equation}  \label{integrable}
||\W_{\Phi(q)}(\infty) 1\!\!1|| = \left|\left| \int_0^\infty\W'_{\Phi(q)}(y)  1\!\!1\, \dd y \right|\right| < \infty \; .
\end{equation}
Note from \eqref{pos}, for $x\geq 0$,
\[
\exp\left(\int_x^\infty \bLambda^*_q (u)\dd u\right)
\]
is a semigroup, i.e. it has entries which are probabilities, Moreover its limit as $x\to\infty$ is the identity matrx.
 Hence
\begin{eqnarray}
\int_0^\infty | [\W'_{\Phi(q)}(y+) 1\!\!1]_{j} | \, \dd y &=&
\lim_{x\to\infty}
\int_0^x | [\W'_{\Phi(q)}(y+) 1\!\!1]_{j} | \, \dd y \nonumber\\
& \leq  & \lim_{x\to\infty} \int_0^x \left[\frac{\dd}{\dd y} \left( \exp\left(\int_y^\infty \bLambda^*_q (u)\dd u\right)\right)1\!\!1\right]_j \, \dd y \; ||\W_{\Phi(q)}(\infty) 1\!\!1|| \nonumber \\
&=&\lim_{x\to\infty}\left[ \exp\left(\int_x^\infty \bLambda^*_q (u)\dd u\right)1\!\!1 \right]_j ||\W_{\Phi(q)}(\infty) 1\!\!1|| \nonumber \\
 & = & ||\W_{\Phi(q)}(\infty) 1\!\!1||< \infty \; . \label{new}
\end{eqnarray}
by \eqref{integrable}.
  On the other hand, the absolute value of the numerator of \eqref{ratio}  is bounded by
\begin{equation}
\int_0^\infty 1_{[0,x]}(y)\, {\rm e}^{\Phi(q)(y-x)}
\left|\left[\bDelta _{\boldsymbol{v}}\left( \Phi(q)\right)
\W'_{\Phi(q)} (y+)\bDelta _{\boldsymbol{v}}\left( \Phi(q)\right)^{-1}1\!\!1 \right]_j\right|
\dd y \, \,
\label{delt}
\end{equation}
Thanks to the fact that  $\bDelta _{\boldsymbol{v}}\left( \Phi(q)\right)$ and $\bDelta _{\boldsymbol{v}}\left( \Phi(q)\right)^{-1}$ are diagonal with bounded positive entries, up to a multiplicative constant, the integrand in \eqref{delt} can be bounded by the integrand on the  left hand side of \eqref{new}.
The limit \eqref{ratio} now holds by dominated convergence theorem in view of \eqref{new}.
\end{proof}

\begin{proof}[Proof  (of Theorem \ref{russian-alternative})]  By Lemma \ref{boringconditions} and Remark \ref{taugrem} in the Appendix, we see that conditions (ii) and (iv) of Theorem \ref{main} hold. Conditions (iii) and (v) are moot.  Following the remarks preceding the statement of Theorem \ref{russian-alternative}, taking note of Remark \ref{dontneedgdiffeq}, we are thus left with verifying the replacement of condition (i) of Theorem \ref{main}, in the form \eqref{gODE2}, which requires that
\begin{equation}
g'(s,j)\leq  1 -  \frac{[\Z^{(q)} (g(s,j)) 1\!\!1]_{j}}{ [ \Z^{(q)\prime} (g(s,j))1\!\!1]_{j}},\qquad j\in E
\label{Russiang}
\end{equation}
also keeping in mind that
\begin{equation} \label{aj}
    g(s,j) \le a(j) ,\qquad j\in E
\end{equation} where $a(j)$ was defined in Corollary \ref{rem2}.
In fact, we are going to abandon checking condition (i) of Theorem \ref{main} and verify instead the weaker condition \eqref{Russiang}, that is implied by that condition in the proof of Theorem \ref{main}.

\bigskip

Taking inspiration from the solution to the Shepp-Shiryaev optimal stopping problem when driven by a spectrally negative L\'evy process, see e.g. \cite{shep1, shep2, avra}, we look for a solution to the system \eqref{Russiang} which takes the form $g(s,j) = c_j$, where $c_j$ are constants in $[0,\infty]$, for each $j\in E$, and compare $c_j$ with $a(j)$. In that case, \eqref{Russiang} tells us that we are looking for $x$ satisfying
\begin{equation}
u_j(x) = [\Z^{(q)} 1\!\!1]_{j}(x) -q [ \W^{(q)}1\!\!1]_{j}(x) \leq 0, \qquad x\leq 0.
\label{roots}
\end{equation}

Remark \ref{necessary} indicates that we should necessarily seek a minimal solution to \eqref{Russiang} and hence we should seek the smallest solution of \eqref{roots}. As such, $c_j$ would then satisfy the definition $c_j = \inf\{x\geq 0: u_j(x)\leq 0\}$ as long as it also satisfies \eqref{aj}.

%

\bigskip

We will proceed with our analysis in the  settings given in the statement of the theorem.
For the three cases (a)-(c),
we have that  $\Phi(q)>1$ (equiv. $q>\kappa(1)$).
Then, by Proposition \ref{prop3}, we get
\begin{equation}
\lim_{x\rightarrow \infty} \frac{u_j (x)}{q[\W^{(q)}   1\!\!1]_j(x)}=\frac{1}{\Phi(q)}-1< 0 \; .
\label{limit}
\end{equation}

As such, one way to verify that a root in the sense of \eqref{roots} exists is to consider the  value of $u_j(0+)$ and the sign of $\lim_{x\rightarrow \infty} [\W^{(q)}   1\!\!1]_j(x).$
\bigskip

First, consider the case  $0\le [\W^{(q)}1\!\!1]_j(0+)<q^{-1}$. When $j\in E^{ubv}$, then $[\W^{(q)}1\!\!1]_j(0+)=0$ and we have $u_j(0+)=1$. On the other hand, when $j\in E^{bv}$ and $[\W^{(q)}1\!\!1]_j(0+)<q^{-1}$, then also on account of the fact that $[\Z^{(q)}1\!\!1]_j(0+)=1$, it follows that $u_j (0+)>0$. Also,
  $\lim_{x\rightarrow \infty} u_j(x)=-\infty$   by the assumption $\lim_{x\rightarrow \infty} [\W^{(q)}   1\!\!1]_j(x)=+\infty$ and \eqref{limit}.  This ensures that there is at least one root $c_j$ in $(0,\infty)$ and $\tau_c$ is the optimal stopping time under the assumption that  $c_j \le a(j)$.
When  $[\W^{(q)}1\!\!1]_j(0+)\ge q^{-1}$, we note that it is automatically the case that $u_j(0+)\leq 0$. In that case, \eqref{Russiang}  is satisfied by taking $c_j = 0$ and it is optimal as $a(j)>0$ and $c_j<a(j)$ holds trivially.

\bigskip
\end{proof}


\begin{prop}
When $0<q\leq \kappa(1)$, the  right-hand side of \eqref{OST} is unbounded, which cannot be attained in an almost surely finite time.
\end{prop}
\begin{proof}
 When $\Phi(q) \leq  1$, for all almost surely finite stopping times $\tau$, we can work with the estimate
\begin{align}
\mathbb{E}_{(x,s, i,j)}[{\rm e}^{-q\tau}{\rm e}^{\overline{X}_\tau}h_{\bar{J}_\tau}]
&\geq\underline{h}\frac{\underline{v}(1)}{\overline{v}(1)} \mathbb{E}_{(x,s, i,j)}\left[{\rm e}^{-\kappa(1)\tau}{\rm e}^{{X}_\tau}{\rm e}^{\overline{X}_\tau - X_\tau}\frac{v_{J_\tau}(1)}{v_i(1)}{\rm e}^{-(q-\kappa(1))\tau}\right]\notag\\
& \geq  \underline{h}\frac{\underline{v}(1)}{\overline{v}(1)} \mathbb{E}_{(x,s, i,j)}^1\left[{\rm e}^{\overline{X}_\tau - X_\tau}\right], \qquad t\geq0,
\label{liminf}
\end{align}
where  $\underline{h} = \min_i h_i$,  $\underline{v}(1) = \min_{i}v_i(1)$ and $\overline{v}(1) = v_i(1)$.
On account of the fact that $\kappa(1)\geq q>0$, the process $(X, J)$ under $\mathbb{P}^1$ is such that the ordinate drifts to $+\infty$, and hence the point $0$ is recurrent for $(\overline{X}-X)$. Now consider stopping time $\tau = \inf\{t>0: \overline{X}_t - X_t>k\}$ for any $k$, which, from Remark \ref{taugrem} in the Appendix, shows that $\tau$ is a $\mathbb{P}$-almost surely finite stopping time. In \eqref{liminf} we see that
\[
\mathbb{E}_{(x,s, i,j)}[{\rm e}^{-q\tau}{\rm e}^{\overline{X}_\tau}h_{\bar{J}_\tau}]\geq\underline{h}\frac{\underline{v}(1)}{\overline{v}(1)} {\rm e}^{k}
\]
and hence, since $k$ can be taken arbitrarily large, it follows that
\[
\sup_\tau \mathbb{E}_{(x,s, i,j)}[{\rm e}^{-q\tau}{\rm e}^{\overline{X}_\tau}h_{\bar{J}_\tau}] = \infty,
\]
where the supremum is taken over all $\mathbb{P}$-almost surely finite stopping times, and the supremum is attained by never stopping.
\end{proof}

Finally, we consider two examples that have appeared in related work. In \cite[Sec.7]{kypal}, the scale matrix is identified for a MAP consisting of  Wiener processes and an independent Markov chain with no intermediate jumps as
\[
\W^{(q)}(x)=\HH {\rm diag}\; (\sinh(x\sqrt{2(q-\lambda_{1})}\,),   \ldots, \sinh(x\sqrt{2(q-\lambda_{N})}\,)\,)  \, \HH^{-1} \; ,
\]
where $\HH$ is the matrix of eigenvectors of $\Q$ and $\lambda_i$ are the corresponding eigenvalues. Hence, $Q=\HH {\rm diag}(\lambda_i) \HH^{-1}  $. For $N=2$, we simply write $q_{1,2} = q_1$, $q_{2,1}=q_2$ and observe that $\lambda_1=0$ and $\lambda_2=-(q_1+q_2)$ with eigenvectors $(1,1)$ and $(q_1,-q_2)$, respectively.
Then, the scale matrix is given by
\begin{eqnarray*}
\W^{(q)}(x)&=&\frac{1}{q_1+q_2}\left[\begin{array}{cc}
1 & q_1 \\
1 & -q_2
\end{array}\right]\left[\begin{array}{cc}
\sinh (x\sqrt{2 q }) & 0 \\
0 & \sinh (x\sqrt{2 (q+q_1+q_2) })
\end{array}\right]\left[\begin{array}{cc}
q_2 & q_1 \\
1 & -1
\end{array}\right]  \\
&=&
 \left[\begin{array}{cc}
\frac{q_2\sinh (x\sqrt{2 q })+q_1\sinh (x\sqrt{2 (q+q_1+q_2) })}{q_1+q_2} & \frac{q_1\sinh (x\sqrt{2 q })-q_1\sinh (x\sqrt{2 (q+q_1+q_2) })}{q_1+q_2} \\
\frac{q_2\sinh (x\sqrt{2 q })-q_2\sinh (x\sqrt{2 (q+q_1+q_2) })}{q_1+q_2} & \frac{q_1\sinh (x\sqrt{2 q })+q_2\sinh (x\sqrt{2 (q+q_1+q_2) })}{q_1+q_2}
\end{array}\right]\; .
\end{eqnarray*}
In this case, the row sums of the scale matrix is positive, that is, $[\W^{(q)} 1\!\!1](x)>0$, whereas the individual entries can be negative for some $x>0$.  Moreover, $\lim_{x\rightarrow\infty}[\W^{(q)} 1\!\!1]_j(x)=+\infty$. Then, by Theorem \ref{russian-alternative}, the optimal stopping problem has a solution with $c_j\ge 0$, $j=1,2$, as given there since $a(1)=a(2)=\infty$.

\bigskip

Another example is a two-dimensional spectrally negative MAP  considered in \cite{ivanovs}, where $X^1$ is taken as a Brownian motion with variance 1 and drift -1, $X^2$ as a compound Poisson process with arrival rate 1, exponential jumps of rate 3, and deterministic drift 2,  $U_{1,2}=0$, $U_{2,1}$ has an Erlang distribution with 2 phases and rate 2, and the transition rates are $q_{1,2}=3$, $q_{2,1}=1$. The matrix $\bPsi$ of \eqref{e:MAP F}
 is given by
 \begin{eqnarray*}
\bPsi(\beta)= \left[\begin{array}{cc}
 -3-\beta +\frac{\beta^2}{2}& \frac{12}{(2+\beta)^2} \\
 1 &  -2+2\beta +\frac{3}{3+\beta}
\end{array}\right]\; .
\end{eqnarray*}
Note that by Theorem \ref{ZWlem}, $[\W^{(q)} 1\!\!1]_1(0+)= W^{(q)}_1(0+)=0$ and $[\W^{(q)} 1\!\!1]_2(0+)= W^{(q)}_2(0+)=1/2$ in view of \cite[Lem.3.1]{KKR}.
For $q=1.5$, we have plotted $[\W^{(q)} 1\!\!1]_j(x)$ and $u_j(x)$ in Fig.\ref{fig:q1.5_1} and Fig.\ref{fig:q1.5_2} for $j=1$ and $j=2$, respectively. Note that for both $j=1,2$, $[\W^{(q)} 1\!\!1]_j(0+) \in [0,q^{-1})$ holds as in case (ii) of Theorem \ref{russian-alternative}.
Since  $[\W^{(q)} 1\!\!1]_1(x)$ is increasing and positive, we predict that $a(1)=\infty$ and we find that $c_1=0.26$. On the other hand, $[\W^{(q)} 1\!\!1]_2(x)$ becomes negative for $x>0.87$, which implies that $0.87<a(2)<\infty$. Although $[\W^{(q)} 1\!\!1]_2(x)$ may not tend to $\infty$, if $u_2(x)=0$ for some $x>0$, we could identify $c_2 \in (0,\infty)$ and argue about optimality   as in the Proof of Theorem \ref{russian-alternative}.
However, we cannot conclude about the solution of the optimal stopping problem (\ref{OST}) since it seems that $c_2=\infty$ from Fig.\ref{fig:q1.5_2} (b). 

\bigskip

When $q=1.8$, again $[\W^{(q)} 1\!\!1]_j(0+) \in [0,q^{-1})$, $j=1,2$, holds and the behavior of  $[\W^{(q)} 1\!\!1]_1(x)$ is similar to $q=1.5$ with $a(1)=\infty$ and
$c_1=0.23$. Moreover, the graph of $[\W^{(q)} 1\!\!1]_2(x)$ and $u_2(x)$ given in Fig.\ref{fig:q1.8_2} show that $0.88 <a(2)<\infty$ and $c_2= 0.17<a(2)$. Therefore, the  optimal stopping problem (\ref{OST}) has a solution with $\tau_{c}=\inf\{t\ge 0:\overline{X}_{t}-X_t\ge c_{\bar{J}_t}\}$, similarly to Theorem  \ref{russian-alternative}, where a sufficient condition for the existence of a root $c_j\in (0, \infty)$ is provided in case (ii).

\bigskip
For $q=5$, the relevant functions are plotted in Fig.\ref{fig:q5_1} and Fig.\ref{fig:q5_2}. In this case, (ii) and (i)  of Theorem \ref{russian-alternative} are applicable to $j=1$ and $j=2$, respectively. It seems that $a(1)=\infty$ in this case as well and  $1.2<a(2)<\infty$. We find that $c_1=0.1$ and $c_2=0$. It follows that there is a solution to the optimal stopping problem.

\begin{figure}[h!]
    \centering
    \includegraphics[width = 8.1cm]{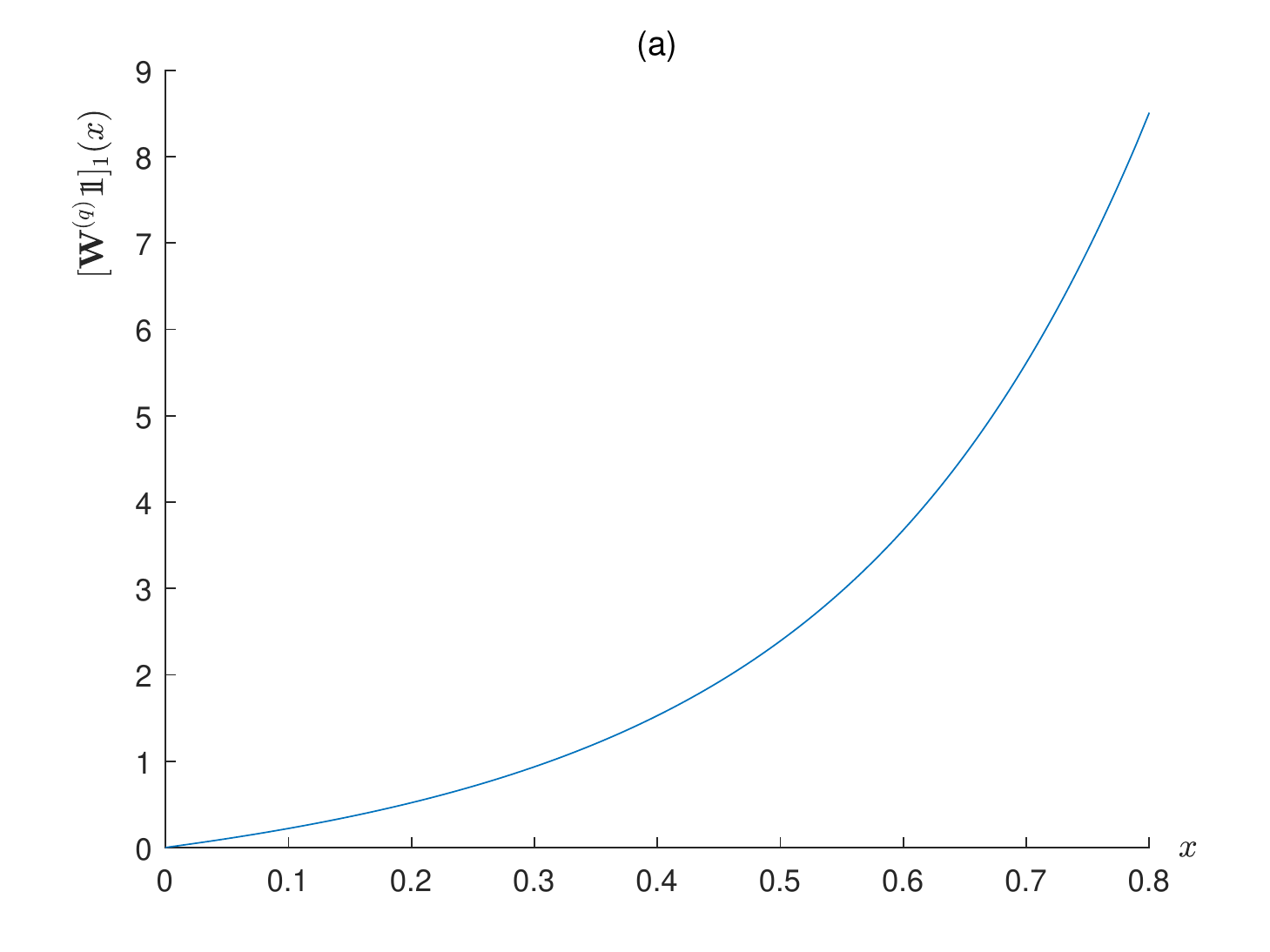}
    \includegraphics[width = 8.1cm]{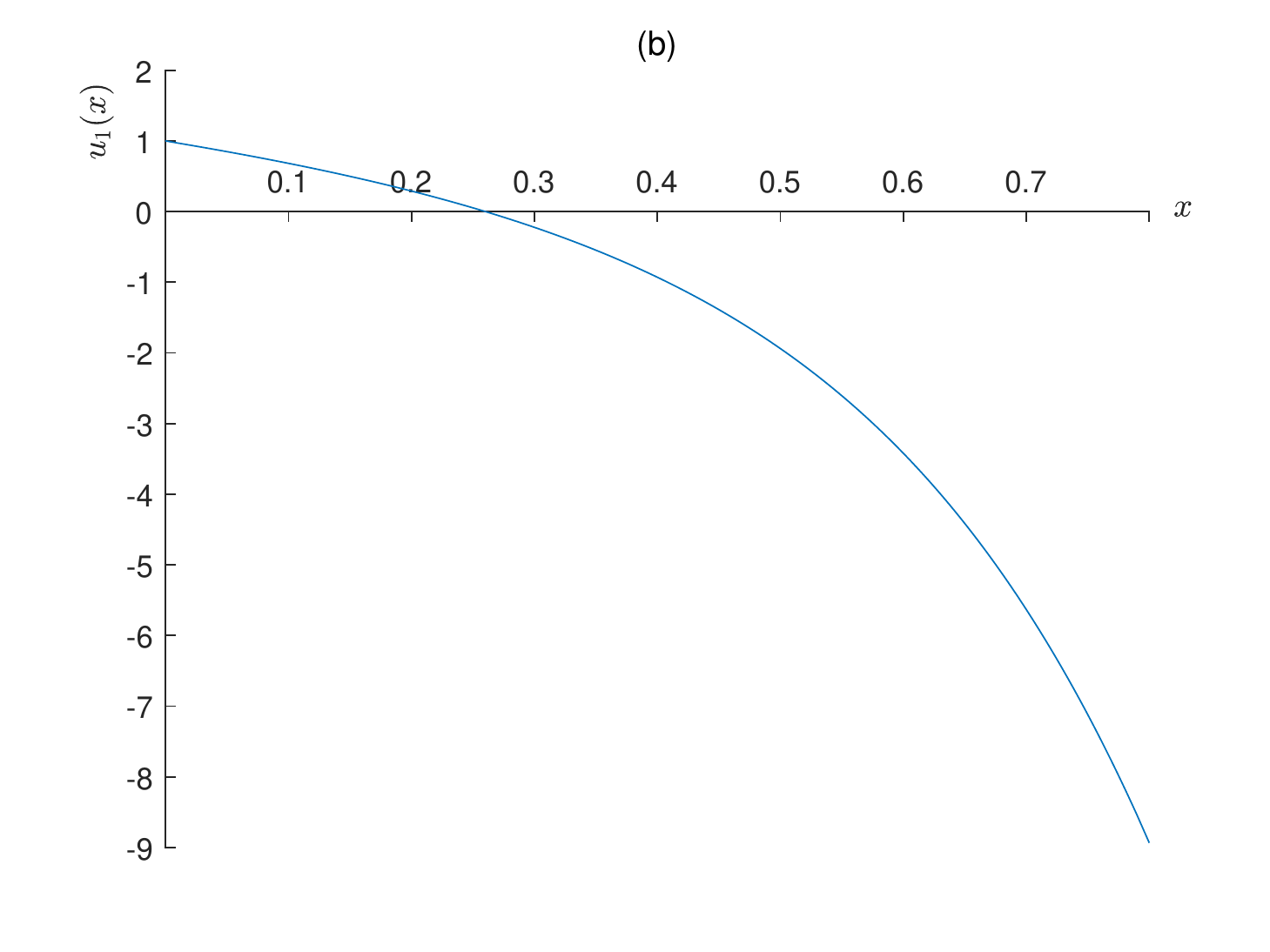}
    \caption{For $q=1.5$, (a)  $[\W^{(q)} 1\!\!1]_1(x)$,  (b) $u_1(x)=[\Z^{(q)} 1\!\!1]_{1}(x) -q [ \W^{(q)}1\!\!1]_{1}(x)$  
}
    \label{fig:q1.5_1}
\end{figure}

\begin{figure}[h!]
    \centering
    \includegraphics[width = 8.1cm]{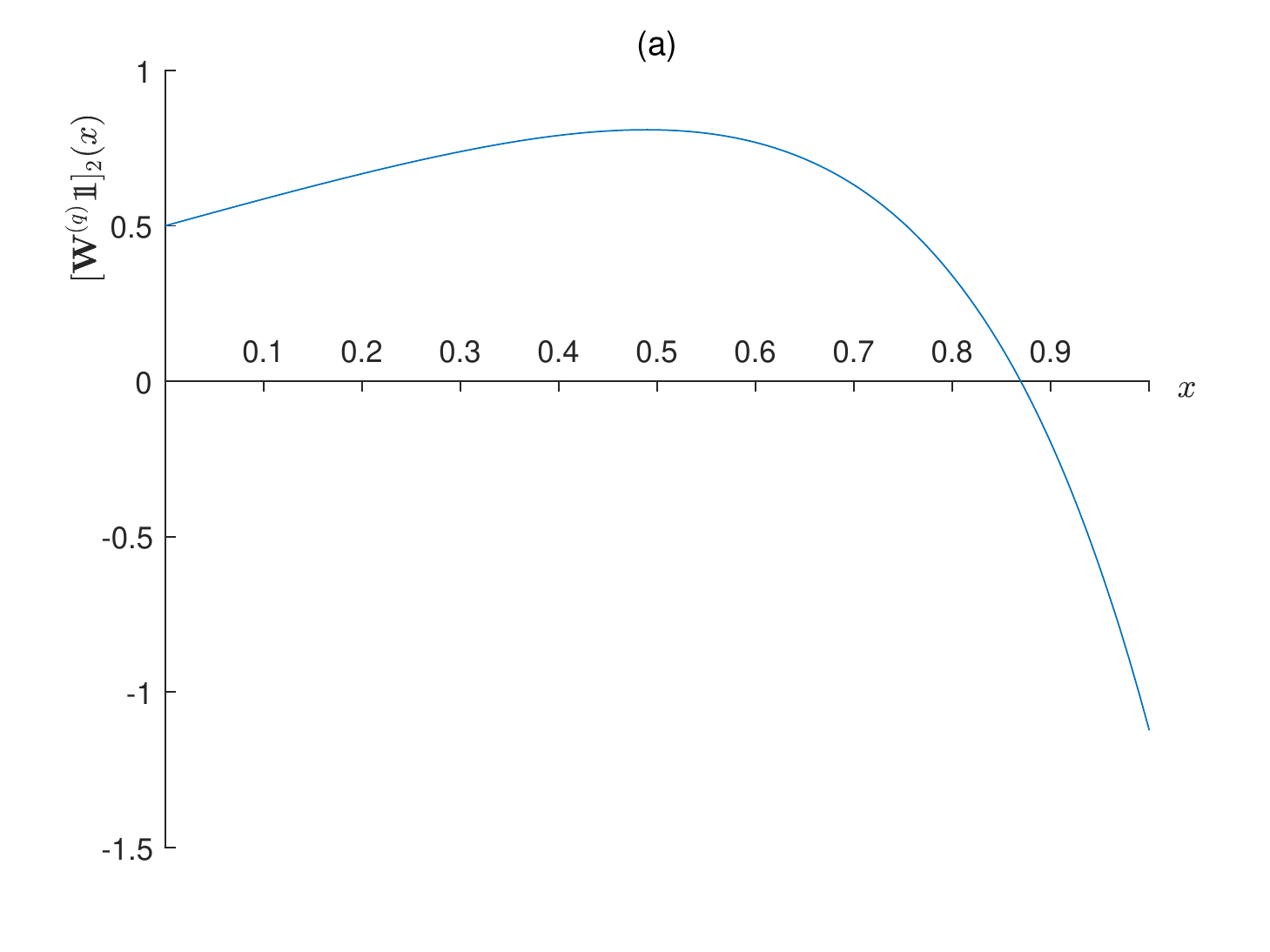}
    \includegraphics[width = 8.1cm]{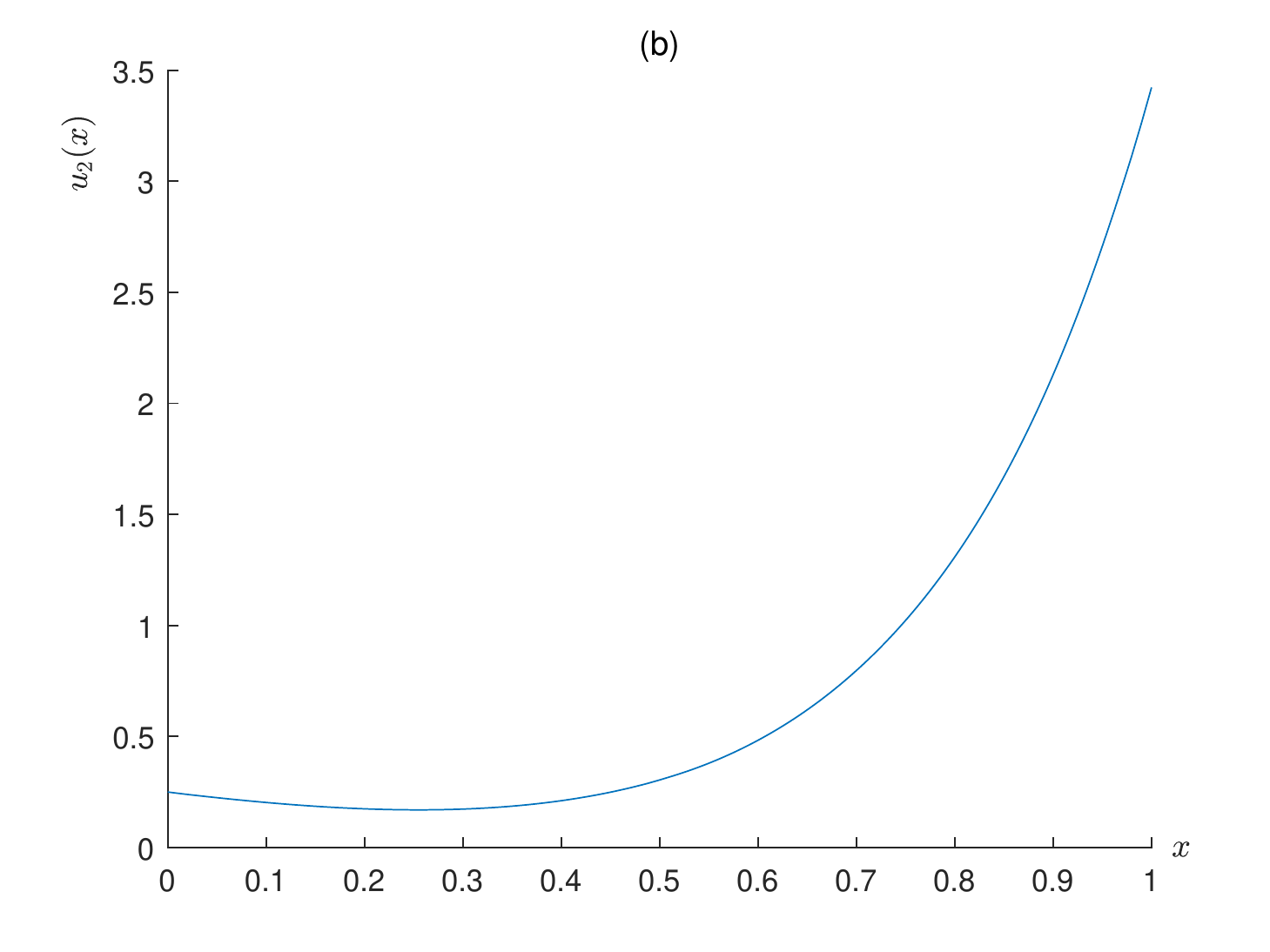}
    \caption{For  $q=1.5$, (a)  $[\W^{(q)} 1\!\!1]_2(x)$,  (b) $u_2(x)=[\Z^{(q)} 1\!\!1]_{2}(x) -q [ \W^{(q)}1\!\!1]_{2}(x)$ 
}
    \label{fig:q1.5_2}
\end{figure}

\begin{figure}[h!]
    \centering
    \includegraphics[width = 8.1cm]{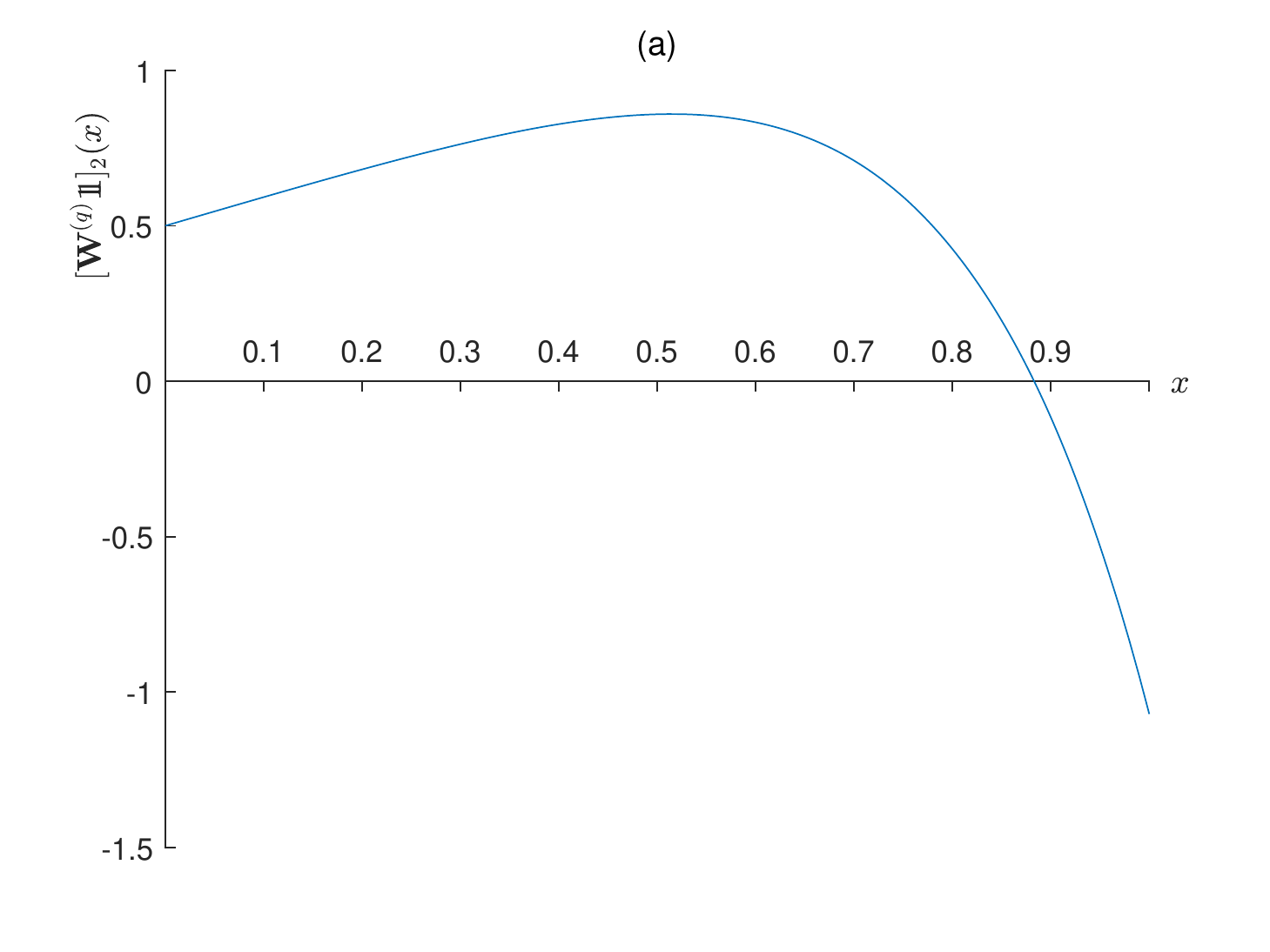}
    \includegraphics[width = 8.1cm]{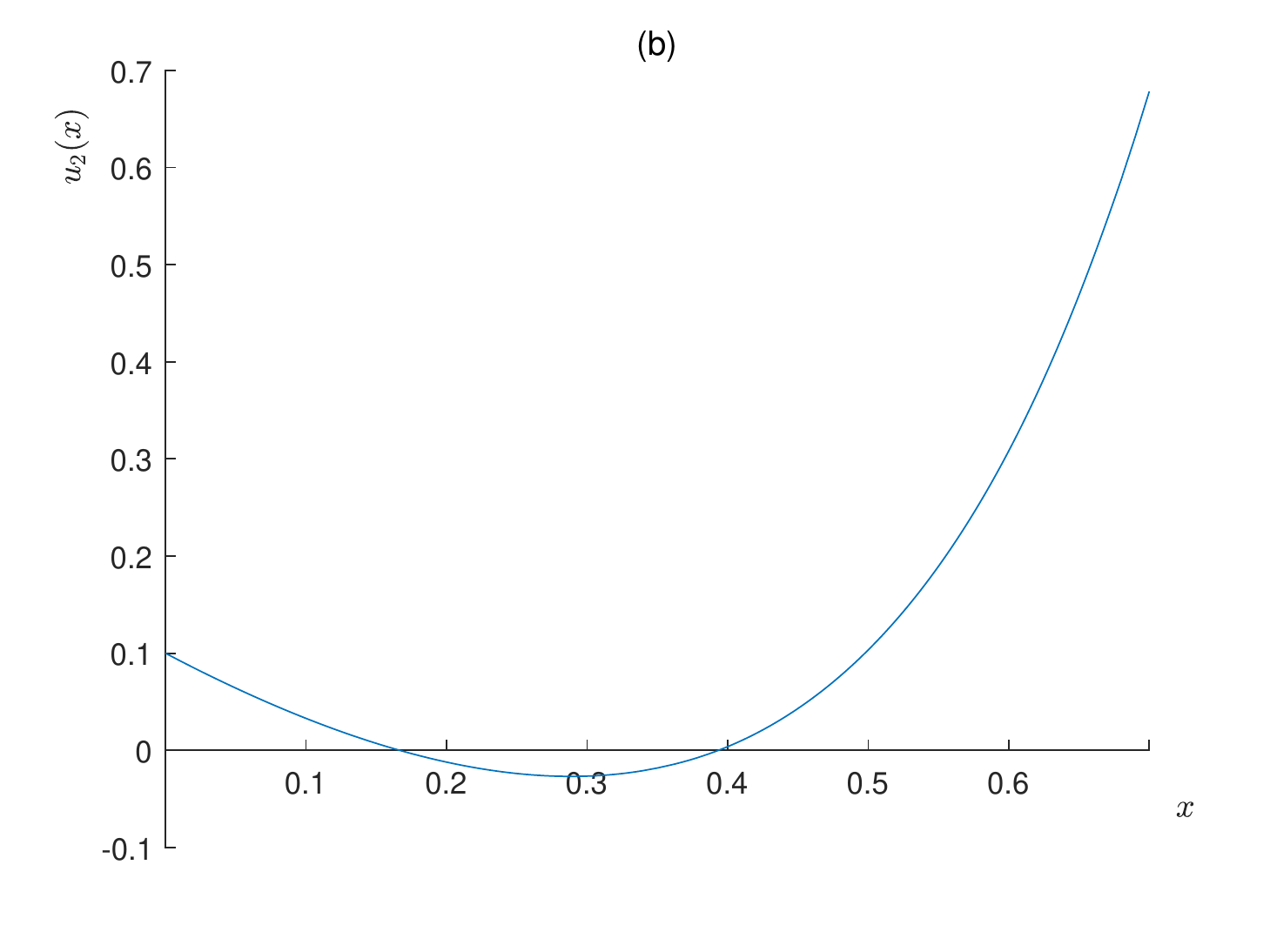}
    \caption{For  $q=1.8$, (a)  $[\W^{(q)} 1\!\!1]_2(x)$,  (b) $u_2(x)=[\Z^{(q)} 1\!\!1]_{2}(x) -q [ \W^{(q)}1\!\!1]_{2}(x)$ 
}
    \label{fig:q1.8_2} 
\end{figure}

\begin{figure}[h!]
    \centering
    \includegraphics[width = 8.1cm]{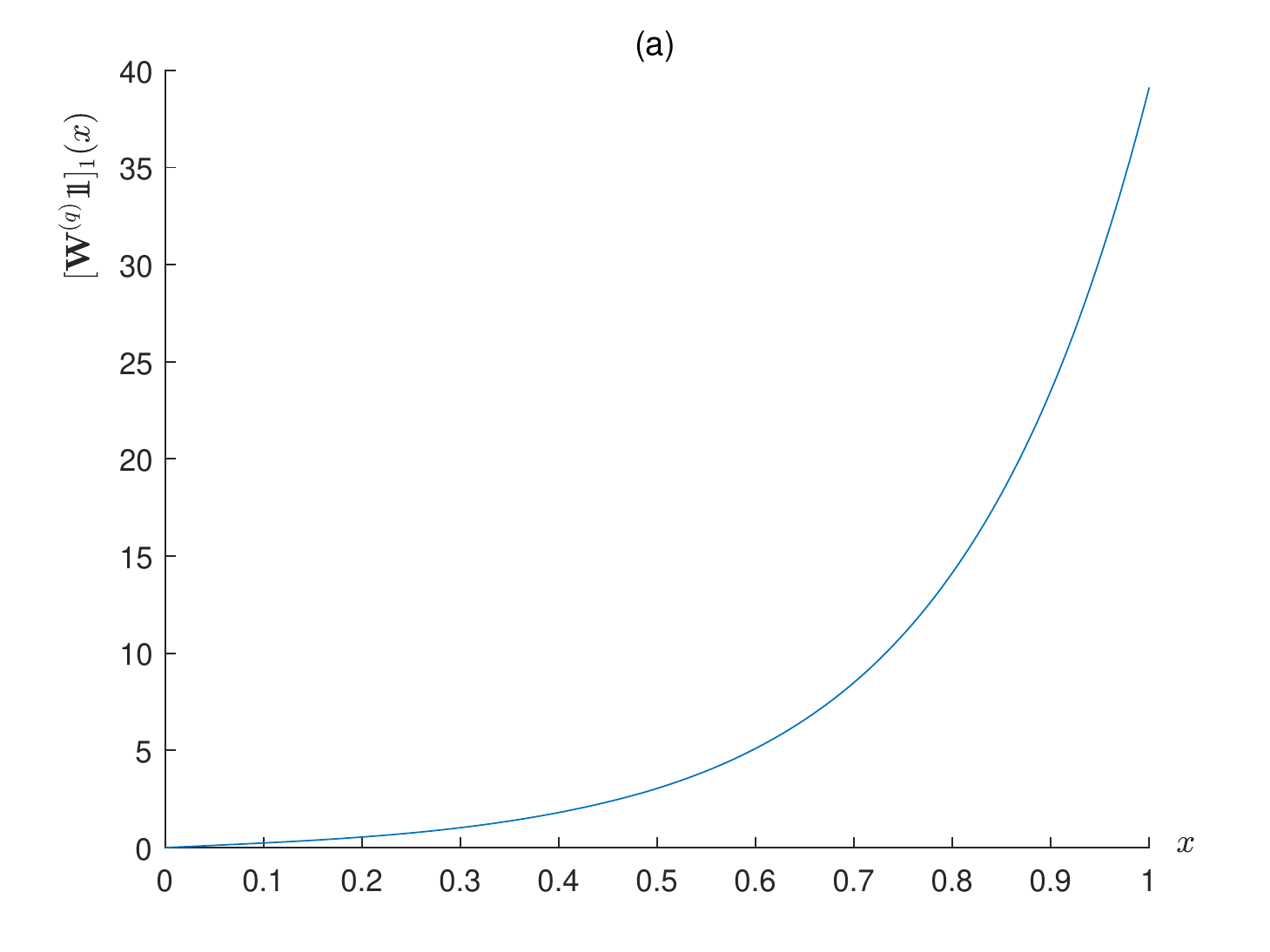}
     \includegraphics[width = 8.1cm]{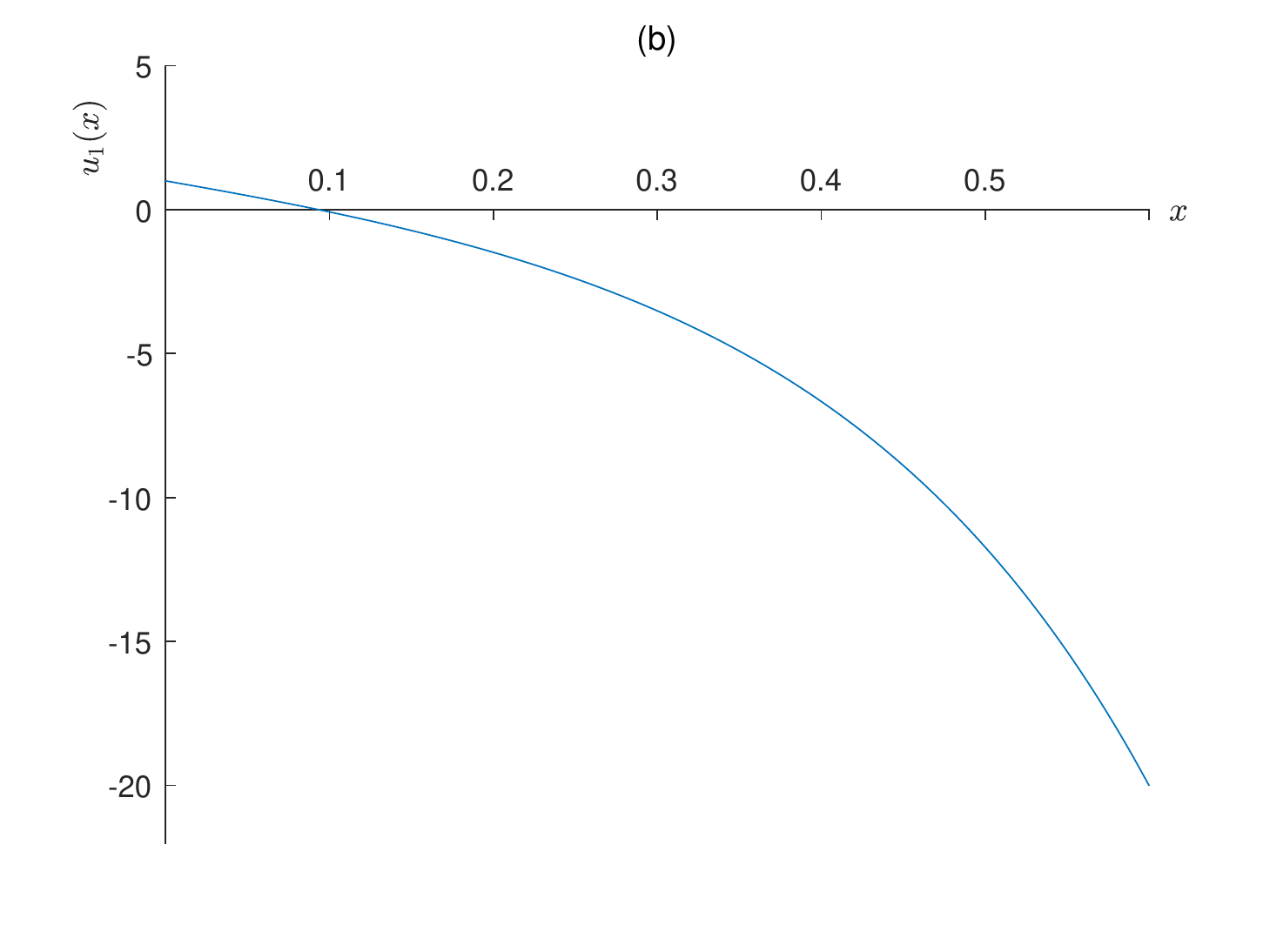}
    \caption{For $q=5$, (a)  $[\W^{(q)} 1\!\!1]_1(x)$, (b) $u_1(x)=[\Z^{(q)} 1\!\!1]_{1}(x) -q [ \W^{(q)}1\!\!1]_{1}(x)$ 
}
    \label{fig:q5_1}
\end{figure}

\begin{figure}[h!]
    \centering
    \includegraphics[width = 8.1cm]{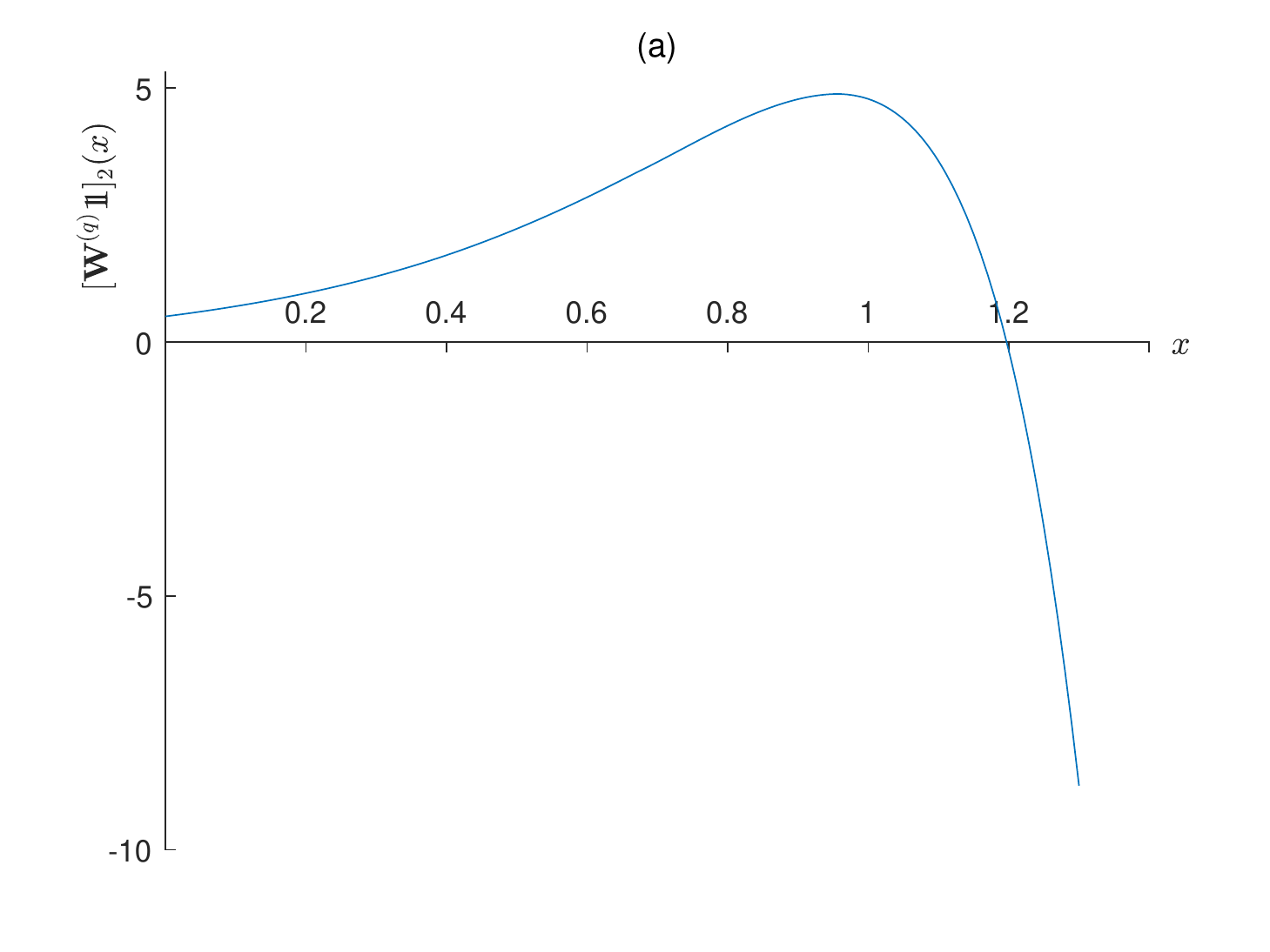}
     \includegraphics[width = 8.1cm]{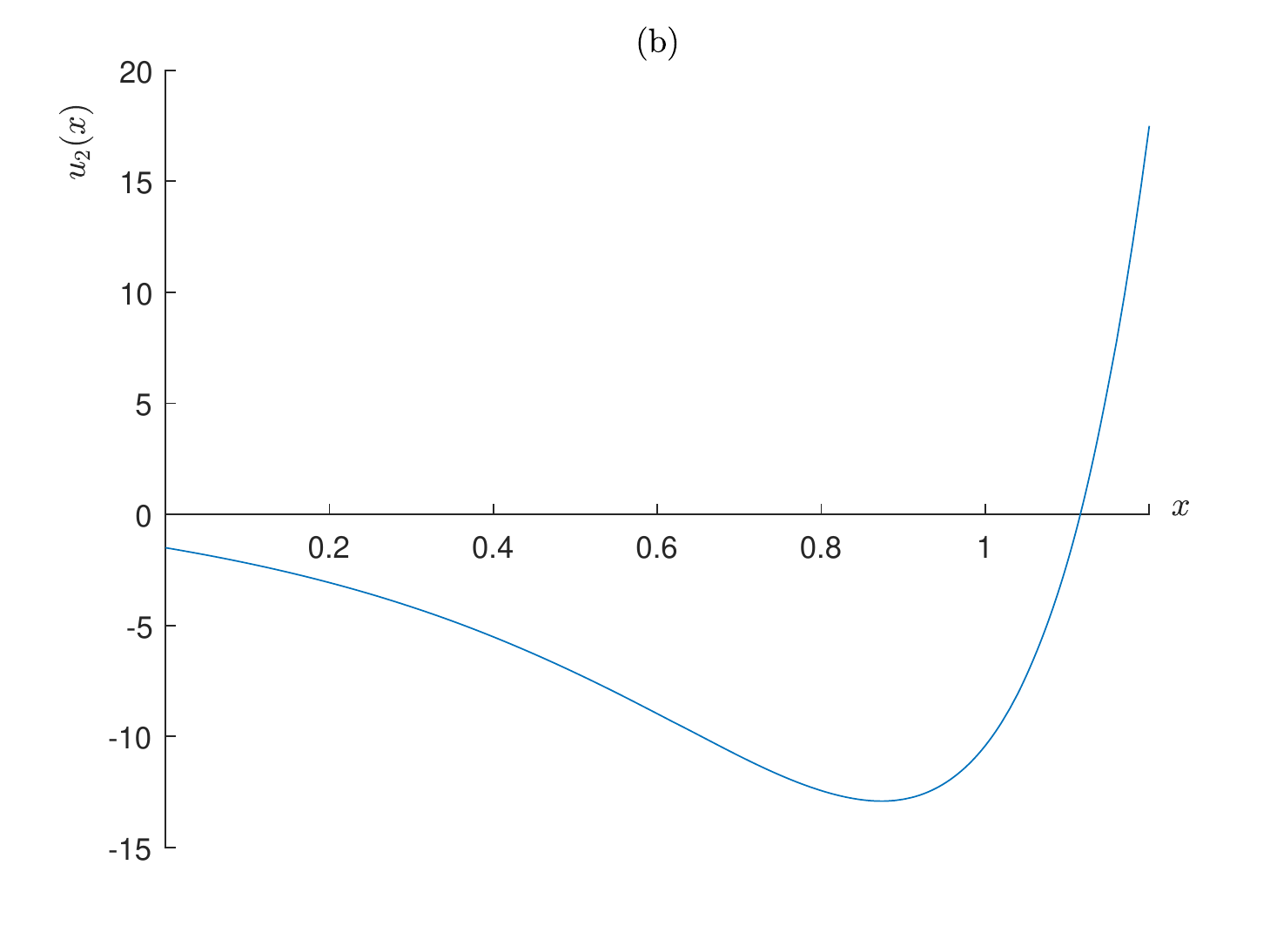}
    \caption{For $q=5$, (a)  $[\W^{(q)} 1\!\!1]_2(x)$, (b) $u_2(x)=[\Z^{(q)} 1\!\!1]_{2}(x) -q [ \W^{(q)}1\!\!1]_{2}(x)$ 
}
    \label{fig:q5_2}
\end{figure}

\section{Conclusion}

In  \cite{curdin} and \cite{KO}, an optimal stopping problem with gain function of the form $({\rm e}^{s\wedge\epsilon}-K)^+$ and discounting was considered for spectrally negative L\'evy processes. There, the authors developed the optimal solution in the form of a non-trivial threshold strategy.  In the current setting, the natural analogue of the gain function takes the form
\[
f(s,j)=({\rm e}^{s\wedge\epsilon}-K)^+ \, h_j, \qquad s\in\mathbb{R}, j\in E, \epsilon >\log(K).
\]
Note that we may effectively take  $s^*:= \log(K)$. To find a solution in the spirit of the aforementioned results for L\'evy processes, amongst other things, Theorem \ref{main} requires us to verify the existence of  a solution to
\begin{equation} \label{dif2}
g^\prime(s,j) =
  1- \frac{ {\rm e}^{s\wedge \epsilon} [\Z^{(q)} (g(s,j))  1\!\!1]_j}{ ({\rm e}^{s\wedge \epsilon} -K)[q\W^{(q)}(g(s,j))1\!\!1]_j }, \qquad s\geq \log (K).
\end{equation}
Whilst we offer no details here, given the technical exploration of the Shepp-Shiryaev optimal stopping problem given in this article, we claim that the interested reader will find the analysis in \cite{KO} is robust enough to carry over to the current setting with a number of straightforward technical modifications. Indeed, the fact that the curves $g(s,j)$, for each $j\in E$, can be phrased through autonomously differential equations (involving only dependency on state $j$) allows one to treat them within the same framework of isocline analysis as in \cite{KO}.
The same can in principle be said of the optimal stopping problems considered in \cite{O}.
\bigskip

As such, the conclusion to this paper is that, modulo some technical adaptation,  the analysis that has led to a relatively wealthy base of results on optimal stopping problems, stochastic games and stochastic control problems driven by spectrally negative L\'evy processes, can now be adapted to handle analogous problems driven by spectrally negative MAPs. See \cite{av2} which summarises a number of identities in this spirit which have the potential to be brought forward into the Markov additive setting.

\section*{Appendix: Proof of Theorem \ref{ZWlem}}\label{ZW}

(i) From  \eqref{id1}, we note that
\begin{equation}
\mathbb{E}_{(0,i)}[{\rm e}^{-q\tau^{+}_{a}};\tau_{a}^{+}<\tau_{0}^{-},J_{\tau_{a}^{+}}=j]=[\W^{(q)}(0+)\W^{(q)}(a)^{-1}]_{i,j}.
\label{isitubv}
\end{equation}
In the setting that $i\in E^{\texttt{ubv}}$, the expectation on the left-hand side of \eqref{isitubv} is necessarily 0, as the MAP initially behaves as an independent copy of a L\'evy process with Laplace exponent $\psi_i$.

\bigskip

When $i\in E^{\texttt{bv}}$ in \eqref{isitubv},  with positive probability, the modulator chain $J$ will remain in state $i$ for long enough that  the L\'evy process with Laplace exponent $\psi_i$ will reach $a/2$ before (necessarily) jumping below the origin,  (see the discussion in Chapter 8 of \cite{kbook} for further details on the path properties of spectrally negative L\'evy processes). {Once the ordinate hits $a/2$, thanks to the non-negativity of $\W^{(q)}(x)$ for $x>0$, see part (i), and the identity \eqref{id1},}  there is a positive probability that the process exits into $(-\infty, 0)$  before $(a,\infty)$  with the modulator in state $j\in E$. It thus follows that the left-hand side of \eqref{isitubv} is strictly positive for $i\in E^{\texttt{bv}}$ and $j\in E$.

\bigskip

Recall that
\[
 {\bPsi}(\beta) = {\rm diag}( \psi_1(\beta) \cdots \psi_N(\beta))
  + {\boldsymbol{Q}} \circ {\boldsymbol{G}}(\beta),
\]
and
$$\int^{\infty}_{0}{\rm e}^{-\beta x}\W^{(q)}( x)\dd x=(\bPsi(\beta)-q\I)^{-1}$$
for $\beta >\max\{{\rm Re}(z): z\in \mathbb{C}, {\rm det}(\bPsi(z)) =q \}$.
By a property of Laplace transform, we have
\begin{equation} \label{W0+}
\W^{(q)}(0+) =\lim_{\beta\rightarrow\infty} \beta (\bPsi(\beta)-q\I)^{-1}\; .
\end{equation}
We start with finding
\[
\lim_{\beta\rightarrow\infty} \frac{1}{\beta} (\bPsi(\beta)-q\I)
\]
and invert this limit to get \eqref{W0+}. The diagonal entries of the matrix $\bPsi(\beta)-q\I$ are in the form $\psi_i(\beta)-q_i-q$, where $q_i=\sum_{j\neq i} q_{i,j}$, and the off-diagonal entries are given by $ q_{i,j} \mathbb{E}({\rm e}^{\beta\, U_{i,j}})$. For the diagonal entries, we have
\[
\lim_{\beta\rightarrow\infty} \frac{1}{\beta} (\psi_i(\beta)-q_i-q) = \lim_{\beta\rightarrow\infty} \frac{1}{\beta} \psi_i(\beta) = \frac{1}{W^{(q)}_i(0+)}
\]
which is finite when $i\in E^{\rm{bv}}$, and infinite for $i\in E^{\rm{ubv}}$, as  $W^{(q)}_i(0+)=0$ in that case. On the other hand, the limit is 0 for the off-diagonal entries because
$
\lim_{\beta\rightarrow\infty}   \mathbb{E}({\rm e}^{\beta\, U_{i,j}})/\beta=0$
as $U_{ij}<0$ a.s. Then, the result follows from \eqref{W0+}.

\bigskip

(ii) We recall the observation from \eqref{pos} that $\W_{\Phi(q)}(\infty)$ may have negative entries. Hence, thanks to  \eqref{Wqdecomp} and the non-negativity of $\boldsymbol{v}(\Phi(q))$, the matrix $\W^{(q)}(x)$ may have negative entries.  On the other hand, we can also see from \eqref{pos} and \eqref{Wqdecomp}
\begin{align*}
\W^{(q)}(x) \W^{(q)}(a)^{-1} &={\rm e}^{\Phi(q)(x-a)}\bDelta _{\boldsymbol{v}}\left( \Phi(q)\right)
\W_{\Phi(q)}(x)\W_{\Phi(q)}(a)^{-1}
\bDelta _{\boldsymbol{v}}\left( \Phi(q)\right)^{-1}\\
&={\rm e}^{\Phi(q)(x-a)}\bDelta _{\boldsymbol{v}}\left( \Phi(q)\right)
\exp\left(\int_x^a \bLambda^*_q (y)\dd y\right)
\bDelta _{\boldsymbol{v}}\left( \Phi(q)\right)^{-1},
\end{align*}
which gives almost everywhere differentiability of $\W^{(q)}(x) \W^{(q)}(a)^{-1} $.
\bigskip

To show the strict positivity of the almost everywhere derivative of $\W(x)\W(a)^{-1}$, for $0\leq x\leq a<\infty$, we need to delve into excursion theory.
It is a straightforward exercise to show that $(\overline{X}-X, J)$ is a reflected MAP. We define
$\bar{M}:=\{t\ge 0 : \overline{X}_{t} - X_t=0\}$ and $\bar{M}^{cl}$ its closure in $[0,\infty)$.
Obviously the set $[0,\infty)\backslash \bar{M}^{cl}$ is an open set and can be written as a union of intervals. We use $\bar{G}$ and $\bar{D}$, respectively, to denote the sets of left and right end points of such intervals.
Define $\bar{R}:=\inf\{t>0:\ t\in \bar{M}^{cl}\}$.
Upwards regularity implies that every point in ${E}$ is regular for $\bar{M}$
{in the sense that $\mathbb{P}_{0,\theta}\left(\bar{R}=0\right)=1$ for all $\theta\in{E}$.}
Thus by
\cite[Theorem (4.1)]{M}
there exists a continuous additive functional $t\mapsto\ell_{t}$ of
$(\overline{X}-X, J)$
which is carried by $\{0\}\times {E}$ and a family of kernels $(\texttt{N}_i, i\in E)$  on $\mathbb{U}$, the space of c\`adl\`ag  mappings from $\mathbb{R}$ to  $ [0,\infty)\times E$ with lifetime $\zeta$ and terminal position at $\zeta$ which  is negative  (written canonically  as a co-ordinate sequence $((\epsilon(s), \Theta(s)), s\leq \zeta)$),
satisfying
$\texttt{N}_i(\zeta=0)=\texttt{N}_i(\Theta(0) \not=i)=0$ and $\texttt{N}_i(1-\mathrm{e}^{-\zeta})\le 1$, for $i\in E$, such that we have the exit system
\begin{equation}\label{M1}
\mathbb{E}_{(0,i)}\left[\sum_{s\in \bar{G}}\Xi_{s}f\circ\theta_{s}\right]=\mathbb{E}_{(0,i)}\left[\int_{0}^{\infty}\Xi_{s}\, \texttt{N}_{J_s}(f){\rm d}\ell_{s}\right]
\end{equation}
for any
non-negative predictable
process $\Xi$ and
any non-negative function $f$ which is measurable with respect to the natural filtration on $\mathbb{U}$. Moreover, by \cite[Theorem (5.1)]{M},  under $\texttt{N}_i$, the process $((\epsilon(s), \Theta(s)), s\leq \zeta)$ has the strong Markov property.
\bigskip

The exit system $((\texttt{N}_i, i\in E), \ell)$ is not unique. A different choice of $\ell$ will result in a different choice of $(\texttt{N}_i, i\in E)$. A convenient choice of $\ell$ that we will work with is $\ell = \overline{X}$.

\bigskip

With this excursion theory in hand, we can identify (in the spirit of the representation of the scale function for L\'evy processes found in Chapter VII of \cite{bertbook}), for $x\leq a$,
\begin{align}
[\W(x)\W(a)^{-1}]_{i,j}&=\mathbb{P}_{(x,i)}(\tau^+_a <\tau^-_0; J_{\tau^+_a} = j)
\notag\\
& = \mathbb{E}_{(x,i)}\left[\exp\left(-  \int_x^a \texttt{N}_{J^+_s}(\bar{\epsilon}>s)
\dd s \right); J^+_{a} = j\right],\qquad i,j\in E,
\label{diffable}
\end{align}
where $\bar{\epsilon} = \sup_{s<\zeta}\epsilon(s)$ and $(J^+_s, s\geq 0)$ is the modulator of the ascending ladder MAP such that $J^+_s = J_{\tau^+_s}$, $s\geq 0$ (cf. the Appendix of \cite{DDK}).
Taking account of the fact that $s\mapsto\texttt{N}_{J^+_s}(\bar{\epsilon}>s)$ is right continuous, we now see by dominated convergence that  $[\W(x)\W(a)^{-1}]_{i,j}$ is  almost everywhere differentiable and that its derivative is right-continuous.

\bigskip

For strict positivity of the aforesaid right-continuous derivative, let us introduce the random variable $N_{a,j,k}$ which is the number of times $\overline{X}-X$ exceeds level $k$ in an excursion which begins in state $j$, before $X$ exceeds $a$. From \eqref{M1} we have that, for any $a,k>0$,
\begin{align}
\mathbb{E}_{(0,i)}[N_{a,i,k}]&=\mathbb{E}_{(0,i)}\left[\sum_{s\in \bar{G}}\mathbf{1}_{(s<\tau^+_a)}\mathbf{1}_{(\epsilon(0) = i, \, \bar\epsilon\, > k)}\circ\theta_{s}\right]\notag\\
&=\mathbb{E}_{(0,i)}\left[\int_{0}^{\infty}\mathbf{1}_{(s<\tau^+_a)}\mathbf{1}_{(J_s = i)}\, \texttt{N}_{i}(\bar\epsilon>k){\rm d}\ell_{s}\right] \notag\\
&=
\mathbb{E}_{(0,i)}\left[\int_{0}^{a} \mathbf{1}_{(J^+_s = i)}{\rm d}{s}\right] \texttt{N}_{i}(\bar\epsilon>k)\, ,
\label{E1}
\end{align}
where we have used our choice $\ell = \overline{X}$.
On the other hand, we can lower bound $\mathbb{E}_{(0,i)}[N_{a,i,k}]$ by restricting the probability space to the event that $\{\sigma_1>\tau^+_a\}$ (recall $\sigma_1$ is the first time that $J$ jumps). As such, we have
\begin{align}
\mathbb{E}_{(0,i)}[N_{a,i,k}]&\geq \mathbf{E}^i_0\left[{\rm e}^{-|q_{i,i}|\tau^{i,+}_{a}}
\sum_{s<\tau^{i,+}_a}\mathbf{1}_{(\overline{X}^i_s - X^i_s >k)}\right] \notag\\
&={\rm e}^{-\phi_i(|q_{i,i}|)a}  {_{i}\mathbf{E}}^{\phi_i(|q_{i,i}|)}_0\left[
\sum_{s<\tau^{i,+}_a}\mathbf{1}_{(\overline{X}^i_s - X^i_s >k)}\right]\notag\\
& =  a{\rm e}^{-\phi_i(|q_{i,i}|)a} \times a\texttt{n}_i(\bar{\epsilon}>k)>0,
\label{E2}
\end{align}
where we recall that $(X^i_t, t\geq 0)$ with probabilities $(\mathbf{P}^i_x, x\in \mathbb{R})$ is the L\'evy process with Laplace exponent $\psi_i$, with  $\phi_i$ as  its the right inverse;  $\mathbf{P}^{i\,, \phi_i(|q_{i,i}|)}_0$ is the result of the Esscher change of measure based on the  martingale $\exp(|q_{i,i}|{ X^{i}_{t}} - \phi(|q_{i,i}|) t )$, $t\geq 0$, $\texttt{n}_i$ is the excursion measure for the Poisson point process of excursions of ${\overline{X}^i} -{ X^i}$ away from $0$ and $\tau^{i,+}_a = \inf\{t>0 : X^i_t>a\}$.

\bigskip

Now  comparing \eqref{E1} and \eqref{E2}, we see that $\texttt{N}_{i}(\bar\epsilon>k)>0$ for all $i\in E$ and $k>0$. Referring back to \eqref{diffable}, we can also see that the right-continuous derivative of $[\W(x)\W(a)^{-1}]_{i,j}$ is strictly positive. Statement (ii) is thus proved for $\W(x)\W(a)^{-1}$.

\bigskip
As was mentioned in (19) of \cite{kypal} (cf. \eqref{Wqdecomp}), we can identify, for $0\leq x\leq a$,
\[
\W^{(q)}(x)\W^{(q)}(a)^{-1} =\bDelta _{\boldsymbol{v}}\left( \Phi(q)\right)
{\rm e}^{\Phi(q)x}
\W_{\Phi(q)}(x)\W_{\Phi(q)}(a)^{-1}
\bDelta _{\boldsymbol{v}}\left( \Phi(q)\right)^{-1},
\]
where $\Phi(q)$ was defined in \eqref{kinverse} and $\W_{\Phi(q)}$ plays the role of $\W^{(0)}$ under $(\mathbb{P}^{\Phi(q)}_{(x,i)}, x\in \mathbb{R}, i\in E)$, as defined in \eqref{alaGirsanov}.
This observation thus allows us to pass the conclusion of the previous paragraph to the more general setting of $\W^{(q)}(x)\W^{(q)}(a)^{-1}$.

\bigskip

(iii) Recall the entries of the matrix $\bPsi(\beta)-q\I$ from the proof of (i) above.
As $\W^{(q)\prime}$ exists a.e. by (ii), we can use its Laplace transform given by $\beta (\bPsi(\beta)-q\I)^{-1} -\W^{(q)}(0+) $, to obtain $\W^{(q)\prime}(0+)$ as
\begin{equation}  \label{derivative}
\W^{(q)\prime}(0+)=  \lim_{\beta\rightarrow\infty} \left( \beta^2 (\bPsi(\beta)-q\I)^{-1}-\beta \W^{(q)}(0+) \right)\; .
\end{equation}
Now consider
\begin{equation}  \label{aux}
    \frac{1}{\beta^2} (\bPsi(\beta)-q\I) \; .
\end{equation}
As $\beta\rightarrow\infty$, the limit  is 0 for the off-diagonal entries of  \eqref{aux} because
$
\lim_{\beta\rightarrow\infty}   \mathbb{E}({\rm e}^{\beta\, U_{i,j}})/\beta^2=0$
in view of the fact that $U_{ij}<0$ a.s.
Moreover, the diagonal entries limit to the Gaussian coefficient of the individual Laplace exponents $\psi_i$, $i\in E$, (which may be zero for some entries).  It follows from \eqref{derivative} that $\W^{(q)\prime}(0+)$ is  a diagonal matrix because $\W^{(q)}(0+)$ is diagonal, by (i), and
\begin{equation} \label{inverse}
\lim_{\beta\rightarrow\infty} \beta^2(\bPsi(\beta)-q\I)^{-1},
\end{equation}
is also  diagonal thanks to  the limit  \eqref{aux}.

\bigskip

%
We will consider  the limit of a diagonal entry $i$ in \eqref{derivative} only for $i\in E^{\rm{ubv}}$, in which case, this entry is equal to the  diagonal entry $i$ in \eqref{inverse} as $\W^{(q)}_{i,i}(0+)=0$. The latter can be found as the reciprocal of the diagonal entry of the limit of \eqref{aux}, which is given by
\[
\lim_{\beta\rightarrow\infty} \frac{1}{\beta^2} (\psi_i(\beta)-q_i-q) = \frac{1}{W^{(q+q_i)\prime}_i(0+)}
\]
since  $W^{(q+q_i)}_i(0+)=0$ for $i\in E^{\rm{ubv}}$.

\bigskip

(iv) From the definition of $\Z^{(q)}$, it is clear that
$
\Z^{(q)\prime}(x)=W^{(q)}(x)(q\I-\Q)\mathbf{1}_{(x\geq 0)},
$
and hence $\Z^{(q)}$ is continuously differentiable, except at $0$ where there is the possibility of
$\W^{(q)}(0+)\neq \boldsymbol{0}.$ Similarly $\Z^{(q)\prime\prime}(x) = \W^{(q)\prime}(x)(q\I-\Q)\mathbf{1}_{(x>0)}$  for almost every $x>0$, and right-continuity of $\Z^{(q)\prime\prime}(x)$ is an immediate consequence of part (ii).
\qed

\begin{rem}\label{taugrem}\rm The introduced excursion theory above also gives rise to a simple proof that  the stopping time $\tau_{c}=\inf\{t\ge 0:\overline{X}_{t}-X_t\ge c_{\bar{J}_t}\}$ is $\mathbb{P}$-almost surely finite when $c_j\in(0,\infty)$ for at least one $j\in E$.  Indeed, note that, for all $x\leq s$ and $i,j\in E$,
\[
\mathbb{P}_{(x,s, i, j)}(\tau_c =\infty)  = \mathbb{E}_{(x,s, i, j)}\left[\exp\left(-\int_0^\infty \texttt{N}_{j}(\bar{\epsilon}>c_j)|_{j= J^+_s}\dd s\right)\right]=0.
\]
\end{rem}

\section*{Acknowledgements} AEK and M\c{C} are grateful for reciprocal support they each received from their respective universities in order to explore Bath-Ko\c{c} collaborative research. All authors would like to thank three referees for their extensive remarks which brought about many improvements to the original document.

\end{document}